\documentclass[a4paper,12pt]{article}

\usepackage{amsmath, amssymb, amsthm}
\usepackage{graphics}

\usepackage{psfrag, booktabs}
\usepackage{color, epsfig, float}
\usepackage[notcite,notref,final]{showkeys}
\usepackage[active]{srcltx}
\usepackage[bookmarksopen, colorlinks, linkcolor = blue, urlcolor = red,
            citecolor = red, menucolor = blue]{hyperref}
\usepackage[titletoc]{appendix}

\usepackage{mathtools}

\newtheorem{theorem}{Theorem}[section]
\newtheorem{lemma}[theorem]{Lemma}

\newtheorem{proposition}[theorem]{Proposition}
\newtheorem{remark}[theorem]{Remark}

\newcommand\F{\mathbf{F}}
\newcommand\y{\mathbf{y}}

\newcommand\summ{\textstyle\sum\limits}

\newcommand\R{\mathbb{R}}

\newcommand\N{\mathbb{N}}

\newcommand{\inner}[2]{\langle #1, #2\rangle}

\newcommand\norm[1]{\|#1\|}
\newcommand\Norm[1]{\left\|#1\right\|}
\newcommand\set[1]{\{#1\}}
\newcommand\Set[1]{\left\{#1\right\}}

\begin{document}

\title{On projective Landweber-Kaczmarz methods for solving systems
of nonlinear ill-posed equations}
\setcounter{footnote}{1}

\author{
A.~Leit\~ao$^\dag$
\and
B.F.~Svaiter$^\ddag$
}

\date{\normalsize \mbox{}\\
$\dag$  Department of Mathematics, Federal University of St.\,Catarina, P.O.\,Box 476,
        88040-900 Florian\'opolis, Brazil.
        \href{mailto:acgleitao@gmail.com}{\tt acgleitao@gmail.com} \\
$\ddag$ IMPA, Estrada Dona Castorina 110, 22460-320 Rio de Janeiro, Brazil.
        \href{mailto:benar@impa.br}{\tt benar@impa.br} \\[4ex]
\normalsize \today}

\maketitle

\begin{abstract}
In this article we combine the {\em projective Landweber method},
recently proposed by the authors, with {\em Kaczmarz's method}
for solving systems of non-linear ill-posed equations.  
The underlying assumption used in this work is the tangential cone
condition.
We show that the proposed iteration is a convergent regularization method.
Numerical tests are presented for a non-linear inverse problem related to the
Dirichlet-to-Neumann map, indicating a superior performance of the proposed
method when compared with other well established iterations.
Our preliminary investigation indicates that the resulting iteration is
a promising alternative for computing stable solutions of {\em large scale
systems of nonlinear ill-posed equations}.
\end{abstract}

\medskip


\section{Introduction} \label{sec:intro}

The classical Kaczmarz iteration consisting of cyclic orthogonal projections
was devised in 1937 by the Polish mathematician Stefan Kaczmarz for solving
(large scale) systems of linear equations \cite{Kac37}.
Since then, this method was successfully used for solving ill-posed linear
systems related to several relevant applications, e.g. X-ray Tomography%
\footnote{In the Tomography community, the Kaczmarz method is called ``Algebraic
Reconstruction Technique'' (ART).}
\cite{Her75, Her80, Nat77, Nat86, Nat97, NW01} and Signal Processing
\cite{Byr15, SMSAK13, VKG14}.

In this manuscript we couple the {\em projective Landweber} (PLW) method
\cite{LS15} with the {\em Kaczmarz} method. The resulting iteration,
designated here by {\em projective Landweber-Kaczmarz} (PLWK) method, is a
new cyclic type method for obtaining stable approximate solutions for
systems of nonlinear ill-posed equations. {\color{black} The idea of
  projecting on a separating halfspace, which depends on the noisy level,
  was already used for the case of linear operators in \cite{Nat86, SSL08}.
}

The \textit{inverse problem} we are interested in consists of determining an unknown
quantity $x \in X$ from the set of data $(y_0, \dots, y_{N-1}) \in Y^N$, where
$X$, $Y$ are Hilbert spaces and $N \geq 1$ (the case $y_i \in Y_i$ with possibly
different spaces $Y_0,\ldots,Y_{N-1}$ can be treated analogously). In practical
situations, the exact data are not known. Instead, only approximate measured data
$y_i^\delta \in Y$ are available such that
\begin{equation}\label{eq:noisy-i}
    \norm{ y_i^\delta - y_i } \ \le \ \delta_i \, , \quad i = 0, \dots, N-1 \, ,
\end{equation}
with $\delta_i > 0$ (noise level). We use the notation $\delta :=
(\delta_0, \dots, \delta_{N-1})$.

The finite set of data above is obtained by indirect measurements of the
parameter $x$, this process being described by the model $F_{i}(x) = y_i$,
for $i = 0, \dots, N-1$. Here $F_i: D_i \subset X \to Y$ are ill-posed
operators \cite{Gr07} and $D_i$ are the corresponding domains of definition.
Summarizing, the abstract functional analytical formulation of the inverse
problems under consideration consists in finding $x \in X$ such that
\begin{equation} \label{eq:inv-probl}
    F_{i}(x) \ = \ y_i^\delta \, , \quad i = 0, \dots, N-1 \, .
\end{equation}

Standard methods for the solution of system \eqref{eq:inv-probl} are
based in the use of \textit{Iterative type regularization}~\cite{BK04,
  EHN96, HNS95, KNS08, Lan51} or \textit{Tikhonov type
  regularization}~\cite{EHN96, Mor93, SV89, Tik63b, TA77, Sch93a} after
rewriting (\ref{eq:inv-probl}) as a single equation
\begin{align} \label{eq:single-op}
 \F(x) = \y^\delta , \ \ \ \ \mbox{ with } \ \ \ \
 \F : = ( F_0, \dots, F_{N-1} ): \textstyle\bigcap\limits_{i=0}^{N-1} D_i \to Y^N , \ \ \
 \y^\delta := \big( y_0^\delta, \dots, y_{N-1}^\delta \big) .
\end{align}
A classical and general condition commonly used in the convergence analysis
of these methods is the \emph{Tangent Cone Condition} (TCC)~\cite{HNS95}.
If one resorts to the functional analytical formulation \eqref{eq:single-op},
one has to face the numerical challenges of solving a large scale system of ill-posed
equations \cite{CFKS13}. When applied to \eqref{eq:single-op}, the above mentioned
solution methods become inefficient if $N$ is large or the evaluations of
$F_i(x)$ and $F'_i(x)^*$ are expensive.

An alternative technique for solving system \eqref{eq:inv-probl} in a stable way is to
use {\em Kaczmarz (cyclic) type regularization methods}. This technique was introduced in
\cite{HLS07,HKLS07},  \cite{CHLS08},  \cite{HLR09},  \cite{BKL10},  \cite{MRL14} 
and  \cite{BK06}
for
the Landweber iteration,
the Steepest-Descent iteration,
the Expectation-Maximization iteration,
the Levenberg-Marquardt iteration,
the REGINN-Landweber iteration,
and the Iteratively Regularized Gauss-Newton iteration respectively.

Our aim is to combine the newly proposed Projective Landweber Method~\cite{LS15}
with the Kaczmarz method.
The Projective Landweber Method (PLW) is an iterative type method for solving
\eqref{eq:inv-probl} when $N=1$ and $F_0$ satisfies the TCC.
In each iteration $k$, a half space separating $x_k$ from the solution
set is defined and $x_{k+1}$ is a relaxed projection of $x_k$ onto this set.
The resulting iterative method for solving \ $F_0(x) \, = \, y_0^\delta$ \ can be
written in the form
\begin{equation} \label{eq:plw-iter}
  x_{k+1}^\delta \ := \ x_k^\delta \ - \
  \theta_k \, \lambda_k \,
  F_0'(x_k^\delta)^* \big( F_0(x_k^\delta) - y_0^\delta \big) \, ,
\end{equation}
where $\theta_k \in (0,2)$ is a relaxation parameter and $\lambda_k \geq 0$ gives
the exact projection of $x_k^\delta$ onto $H_{0,x_k^\delta}$ (see \cite[Eq.\,(8)]{LS15}).
Observe that this iteration is a Landweber iteration with a stepsize control.
In the next paragraph we present a combination of the PLW method with Kaczmarz
method, for solving \eqref{eq:inv-probl} when $N>1$.

\subsubsection*{The Projective Landweber Kaczmarz (PLWK) method:}

The PLWK method for the solution of \eqref{eq:inv-probl} proposed in this article
consists in coupling the PLW method \eqref{eq:plw-iter} with the Kaczmarz (cyclic)
strategy and incorporating a bang-bang parameter, namely
\begin{equation} \label{eq:plwk-iter}
  x_{k+1}^\delta \ := \ x_k^\delta \ - \
  \theta_k \, \lambda_k \, \omega_k \,
  F_{[k]}'(x_k^\delta)^* \big( F_{[k]}(x_k^\delta) - y_{[k]}^\delta \big) \, .
\end{equation}
Here the parameters $\theta_k$, $\lambda_k$ have the same meaning as in \eqref{eq:plw-iter}
(see \eqref{def:tau-p-lbd} for the precise definition of $\lambda_k$) while
\begin{equation} \label{def:wk-def}
\omega_k \ = \ \omega_k(\delta_{[k]}, y_{[k]}^\delta) \ := \
\begin{cases}
      1  & \norm{F_{[k]}(x_{k}^\delta) - y_{[k]}^\delta} > \tau \delta_{[k]} \\
      0  & \text{otherwise}
\end{cases} \, ,
\end{equation}
where $\tau > 1$ is an appropriate chosen positive constant \eqref{def:tau-p-lbd}
and  $[k] := (k \mod N) \in \set{0, \dots, N-1}$.
{\color{black}
We also consider PLKWr a ``randomized'' version of the method (in the spirit of~\cite{BB97})
where $[k]$ is randomly chosen in $\set{0,\ldots,N-1}$.
}

As usual in Kaczmarz type algorithms, a group of $N$ subsequent steps (starting at
some integer multiple of $N$) is called a {\bf cycle}. In the case of noisy data,
the iteration terminates if all $\omega_k$ become zero within a cycle, i.e., if \
$\norm{F_{i}(x_{k+i}^\delta) - y_{i}^\delta} \leq \tau \delta_i$, $i\in\set{0,\dots,N-1}$,
for some integer multiple $k$ of $N$.

The PLWK iteration scheme in \eqref{eq:plwk-iter}, \eqref{def:wk-def} exhibits the following
characteristics:

\noindent $\bullet$ For noise free data, $\omega_k = 1$ for all $k$ and each cycle consist
of exactly $N$ steps of type \eqref{eq:plw-iter}. Thus, the numerical effort required for
the computation of one cycle of PLWK rivals the effort needed to compute one step
of PLW (or LW)
for \eqref{eq:single-op}. \\
$\bullet$ In the {\em realistic noisy data case}, the bang-bang relaxation parameter
$\omega_k$ will vanish for some $k$ (especially in the last iterations).
Consequently, the computational evaluation of $F'_{[k]}(x_k^\delta)^*$
might be avoided,
making the PLWK method a fast alternative to conventional regularization techniques for
the single equation approach \eqref{eq:single-op}.

\noindent $\bullet$ The {\em convergence of the residuals in the maximum norm} better
exploits the estimates for the noisy data (\ref{eq:noisy-i}) than the standard
regularization methods for \eqref{eq:single-op}, where only \
$N^{-1} \sum_{i=0}^{N-1} \norm{F_i(x_k^\delta) - y_{i}^\delta}^2$ (the {\em squared
average of the residuals}) falls below a certain threshold.
Moreover, the parameter $\omega_k$ in \eqref{def:wk-def} effects that the iterates
$x_k^\delta$ in \eqref{eq:plwk-iter} become stationary in such a way that
\textit{each residual} \ $\norm{F_i(x_k^\delta) - y_{i}^\delta}$ \ in
\eqref{eq:inv-probl} falls below some threshold. This makes (\ref{eq:plwk-iter})
a convergent regularization method in the sense of \cite{EHN96}.

\subsubsection*{Outline of the article:}

In Section~\ref{sec:nbr} we state the main assumptions and derive some
preliminary results and  estimates.
In Section~\ref{sec:plwk} we define the convex sets $H_{i,x}$ related to the operator
equations in \eqref{eq:inv-probl} and prove a special separation property of these
sets. The PLWK iteration is described in detail and a stopping criteria is defined
(in the noisy data case), which is proved to be finite.
Moreover, the first convergence analysis results are obtained, namely: monotonicity
of the iteration error (Proposition~\ref{pr:monot}) and square summability of iteration
steps \eqref{eq:plw-noise2}.
In Section~\ref{sec:conv-anal} weak convergence of the PLWK method for exact data
is proven. Moreover, stability and semi-convergence results are presented.
Section~\ref{sec:plwkr} is devoted to the investigation of a randomized version of
the PLWK method, here denoted by PLWKr method.
In Section~\ref{sec:numerics} we present numerical experiments for a nonlinear
parameter identification problem related to the Dirichlet-to-Neumann map
\cite{LS15, BKL10, HKLS07, LMZ06a, Le06, BELM04}, while Section~\ref{sec:conclusion}
is devoted to final remarks and conclusions.
In the Appendix a strongly convergent version of the PLWK method for exact data is analyzed.

\section{Main assumptions and auxiliary results}
\label{sec:nbr}

In this section we state our main assumptions and discuss some of their
consequences, which are relevant for the forthcoming analysis. In what
follows, we adopt the simplified notation
\begin{equation}
  \label{eq:fdelta}
  F_{i,\delta}(x) \ := \ F_i(x) - y_i^\delta
  \quad\quad {\rm and }  \quad\quad
  F_{i,0}(x) \ := \ F_i(x) - y_i \, .
\end{equation}

Throughout this work we make the following assumptions, which are
standard in the recent analysis of iterative regularization methods
(cf., e.g., \cite{EHN96, KNS08, Sch93a}):

\begin{description}
\item[A1] Each $F_i$ is a continuous operator defined on $D(F_i) \subset X$,
and the domain $D := \bigcap\nolimits_i \, D(F_i)$ has nonempty interior.
Moreover, the initial iterate $x_0 \in D$ and  
there exist constants $C$, $\rho > 0$ such that
$F_i'$, the Gateaux derivative of $F_i$, is defined on $B_\rho(x_0) \subset D$
and satisfies
\begin{equation} \label{eq:a-dfb}
    \| F_i'(x) \| \ \le \ C \, , \ \ \ x \in B_\rho(x_0) \, , \ \ \
    i = 0, \dots, N-1 \, ;
\end{equation}
\item[A2] The \emph{local tangential cone condition} (TCC) \cite{HNS95, KNS08, EHN96}
\begin{equation} \label{eq:tcc}
  \| F_i(\bar{x}) - F_i(x) -  F_i'(x)( \bar{x} - x ) \|_Y \ \leq \
  \eta \norm{ F_i(\bar{x}) - F_i(x) }_Y \, , \quad \forall \ x, \bar{x} \in B_{\rho}(x_0)
\end{equation}
holds for some $\eta < 1$ and $i = 0, \dots, N-1$;
\item[A3] There exists an element $x^\star \in B_{\rho/2}(x_0)$ such that
$F_i(x^\star) = y_i$, for $i = 0, \dots, N-1$, where $y_i \in Rg(F_i)$ are the
exact data satisfying
\eqref{eq:noisy-i};
\item[A4] All operators $F_i$ are continuously Fr\'echet differentiable on
$B_\rho(x_0)$;
\end{description}

\noindent (in \textbf{A2} -- \textbf{A4} the point $x_0 \in X$ and the constant
$\rho > 0$ are as in  \textbf{A1}).

\medskip

Observe that in the TCC we require $\eta < 1$ (see \cite{LS15}) whereas
in classical convergence analysis for the nonlinear Landweber under this
condition, $\eta<1/2$ is required instead (see \cite{EHN96, KNS08}).

The next proposition contains a collection of auxiliary results and estimates
that follow directly from \textbf{A1} -- \textbf{A3}. For a complete proof we
refer the reader to \cite[Section~2]{LS15}.

\begin{proposition}
  \label{pr:prelim}
  If \textbf{A1} -- \textbf{A3} hold, then for any $x$, $\bar x \in B_\rho(x_0)$,
  and \ $i = 0, \dots, N-1$ we have
  \begin{enumerate}
  \item\label{it:1}
    $(1-\eta)\norm{F_i(x) - F_i(\bar x)} \ \leq \ \norm{F_i'(x) (x - \bar x)}
    \ \leq \ (1+\eta) \norm{F_i(x) - F_i(\bar x)}$.
  \item\label{it:2}
    $\inner{F_i'(x)^* F_{i,0}(x)}{x - \bar x} \ \leq \
    (1+\eta) (\norm{F_{i,0}(x)}^2 + \norm{F_{i,0}(x)} \norm{F_{i,0}(\bar x)})$.
  \item\label{it:3}
    $\inner{F_i'(x)^* F_{i,0}(x)}{x - \bar x} \ \geq \
    (1-\eta) \norm{F_{i,0}(x)}^2 - (1+\eta) \norm{F_{i,0}(x)} \norm{F_{i,0}(\bar x)}$.
  \item\label{it:4}    
  If, additionally, \ $F_{i,0}(x)\neq 0$ \ then
    $$
    (1-\eta) \norm{F_{i,0}(x)} - (1+\eta) \norm{F_{i,0}(\bar x)} \
    \leq \ \norm{F_i'(x)^* (x - \bar x)} \
    \leq \ (1+\eta) (\norm{F_{i,0}(x)} + \norm{F_{i,0}(\bar x)}) .
    $$
  \item\label{it:5}
  $F_{i,0}(x) = 0$ \ if and only if \ $F_i'(x)^* F_{i,0}(x) = 0$.
  \item\label{it:6}
  For any \ $(x_k) \in B_\rho(x_0)$ \ converging to \ $\bar x$, the following statements
  are equivalent: \bigskip
  
  \centerline{a) \ $\lim\limits_{k\to\infty} \norm{F_i'(x_k)^* F_{i,0}(x_k)}\,=\,0$;
      \ \quad b) \ $\lim\limits_{k\to\infty} \norm{F_{i,0}(x_k)} \, = \, 0$;
      \ \quad c) $F_{i,0}(\bar x) \, = \, 0$.}

  \item\label{it:7}
  If \ $x^\star \in B_\rho(x_0) \cap F_{i,0}^{-1}(y)$ \ then \
  $\norm{y_i - y_i^\delta - F_{i,\delta}(x) - F_i'(x)(x^\star-x)} \ \leq \
  \eta \, \norm{y_i - y_i^\delta - F_{i,\delta}(x)}$.
  \end{enumerate}
\end{proposition}

We conclude this section proving that, under the TCC, the graph of each operator
$F_i$ is weak$\times$strong sequentially closed.

\begin{proposition}
  \label{pr:tcc-cl-gr}
Let \textbf{A1} -- \textbf{A2} be satisfied and $i \in \set{0,\dots,N-1}$.
If $(x_k)$ in $B_\rho(x_0)$ converges weakly to some $\bar x$ in $B_\rho(x_0)$
and $(F_i(x_k))$ converges strongly  to $z \in Y$, then $F_i(\bar{x}) = z$.
\end{proposition}

\begin{proof}
It follows from \textbf{A2} that 
\begin{eqnarray*}
\eta^2 \norm{ F_i(x_k) \!\!\! & - & \!\!\! F_i(\bar{x}) }^2
\ \geq \ \| F_i(x_k) - F_i(\bar{x}) -  F_i'(\bar{x})( x_k - \bar{x} ) \|^2  \\
&  = & \| F_i(x_k) - F_i(\bar{x}) \|^2 + \| F_i'(\bar{x})(x_k - \bar{x}) \|^2
       - 2 \, \inner{F_i(x_k) - F_i(\bar{x})}{F_i'(\bar{x})(x_k - \bar{x})}  \\
&\geq& \| F_i(x_k) - F_i(\bar{x}) \|^2
       - 2 \, \inner{F_i(x_k) - F_i(\bar{x})}{F_i'(\bar{x})(x_k - \bar{x})} \, .
\end{eqnarray*}
Consequently,
\begin{eqnarray}
(1-\eta^2) \norm{ F_i(x_k) \!\!\! & - & \!\!\! F_i(\bar{x}) }_Y^2
\ \leq \ 2 \, \inner{F_i(x_k) - F_i(\bar{x})}{F_i'(\bar{x})(x_k - \bar{x})}  \nonumber \\
&=& \ 2 \,\inner{F_i(x_k) -z}{F_i'(\bar{x})(x_k - \bar{x})} 
+2\inner{z- F_i(\bar{x})}{F_i'(\bar{x})(x_k - \bar{x})} 
 \nonumber \\
%
%
& \leq & \ 2 \,\norm{F_i(x_k)-z}C\norm{x_k-\bar x}
+ \inner{F_i'(\bar{x})^* [z - F_i(\bar{x})]}{x_k - \bar{x}} \ 
\label{eq:tcc-cl-gr}
\end{eqnarray}
where the second inequality follows from Cauchy-Schwarz inequality and
\textbf{A1}. Since $F_i(x_k)-z\to 0$, $x_k-\bar x \rightharpoonup 0$ as
$k\to\infty$ (and $(x_k)$ bounded), both terms on the right hand side of
the last inequality converge to zero. By \textbf{A2}, $0<\eta<1$;
therefore, $F_i(x_k)-F_i(\bar x)$ also converges to zero.
\end{proof}

\section{The PLWK method}
\label{sec:plwk}

In this section we describe in detail the PLWK method and its relaxed variants.
A stopping index is defined (in the noisy data case). Additionally, preliminary
convergence results are proven, namely: monotonicity of the iteration error,
square summability of the iterative steps norm (in the exact data case) and
finiteness of the above mentioned stopping index (in the noisy data case).

Define, for each $x \in D$ and $i = 0, \dots, N-1$, the sets
\begin{align} \label{eq:lw.cvx.d}
  H_{i,x} \, := \, \Set{ z \in X \, \Big| \, \inner{z-x}{F_i'(x)^*F_{i,\delta}(x)}
  \ \leq \
  -\norm{F_{i,\delta}(x)} \Big( (1-\eta) \Norm{F_{i,\delta}(x)} - (1+\eta) \delta_i \Big) }.
\end{align}
Notice that $H_{i,x}$ is either an empty set, a closed half-space, or $X$.
The next lemma contains a separation result.

\begin{lemma}[Separation]
  \label{lm:separat}
  Suppose that \textbf{A1} and \textbf{A2} hold. If $x\in B_\rho(x_0)$, then
  for $H_{i,x}$ as in \eqref{eq:lw.cvx.d}
  \begin{align*}
    \set{z\in B_\rho(x_0)\;|\; F_i(z)=y_i} \subset H_{i,x}.
  \end{align*}
  Moreover, if $\norm{F_{i,\delta}(x)}>(1+\eta) (1-\eta)^{-1} \delta_i$
  then $x\notin H_{i,x}$.
\end{lemma}
\begin{proof}
  The first assertion follows from \cite[Lemma~4.1]{LS15} and  \eqref{eq:lw.cvx.d}.
The second assertion follows directly from  \eqref{eq:lw.cvx.d}. 
\end{proof}


\begin{remark}
Two facts related to Lemma~\ref{lm:separat} deserve special attention: \\
$\bullet$ Since $\norm{F_{i,\delta}(x)} > (1+\eta) (1-\eta)^{-1} \delta_i$ is sufficient
for separation of $x$ from $F_i^{-1}(y_i)$ in $B_\rho(x_0)$ via $H_{i,x}$, this
condition also guarantees that $F_i'(x)^* F_{i,\delta}(x) \neq 0$. \\
$\bullet$ In the exact data case (i.e., $\max \{\delta_0,\dots,\delta_{N-1} \} = 0$)
the definition \eqref{eq:lw.cvx.d} reduces to
$H_{i,x} \, := \, \set{ z \in X \ | \ \inner{z-x}{F_i'(x)^*F_{i,0}(x)} \, \leq \,
- (1-\eta) \Norm{F_{i,0}(x)}^2 }$.
Therefore, in this case, we have strict separation, $x\notin H_{i,x}$
whenever $F_i(x) \neq y_i$.
\end{remark}

Let
\begin{subequations}  \label{def:tau-p-lbd}
  \begin{align}
\tau & >  (1+\eta)(1-\eta)^{-1} \, ,  \\
p_{i}(t) & := \ t ( (1-\eta)t - (1+\eta) \delta_i ) \, , & \\
 \lambda_k & := \left\{ \begin{array}{ll}
                  \dfrac{ p_{[k]}( \norm{F_{[k],\delta}(x_k^\delta)} )}
                  { \norm{F_{[k]}'(x_k^\delta)^* F_{[k],\delta}(x_k^\delta)}^2 }  , \
               &   {\rm if } \ F_{[k]}'(x_k^\delta)^* F_{[k],\delta}(x_k^\delta) \neq 0 \\
                  0 & , \ {\rm otherwise}
                \end{array} \right. 
\end{align}
\end{subequations}
for $i \in \set{0, \dots, N-1}$ and $k \geq 0$.%
\footnote{Notice that $F_{[k]}'(x_k^\delta)^* F_{[k],\delta}(x_k^\delta) \neq 0$ iff
$F_{[k],\delta}(x_k^\delta) \neq 0$; see Proposition~\ref{pr:prelim}, item~\ref{it:5}.}
The iteration formula of the PLWK method and its relaxed variants is given by
\eqref{eq:plwk-iter}, \eqref{def:wk-def} with $\tau$ and $\lambda_k$ as in
\eqref{def:tau-p-lbd}.


The (exact) \emph{PLWK method} is obtained by taking $\theta_k=1$ in \eqref{eq:plwk-iter},
which amounts to define $x_{k+1}^\delta$ as the orthogonal projection of $x_k^\delta$ onto
$H_{i,x_k^\delta}$. A \emph{relaxed variant of the PLWK method} uses $\theta_k \in (0,2)$
so that $x_{k+1}^\delta$ is defined as a relaxed projection of $x_k^\delta$ onto
$H_{i,x_k^\delta}$. The computation of the sequence $(x_k^\delta)$ should be stopped
at the index $k_*^\delta \in \N$ defined by
\begin{equation}  \label{def:discrep}
  k_*^\delta \ := \ \min \Big\{ lN \in \N \ | \ x_{lN}^\delta =
                 x_{lN+1}^\delta = \cdots = x_{lN+N}^\delta \Big\} ,
\end{equation}

In what follows $\lfloor k \rfloor$ denotes the biggest integer less or equal to $k$
(notice that $k$ = $\lfloor k/N \rfloor \cdot N + [k]$ for all $k \in \N$).

\begin{remark}
  \label{rem:stop.index}
  Concerning the above definition of the stopping index $k_*^\delta$:
  \medskip
  
  \noindent i) Equivalently, one can define $k_*^\delta$ as the smallest
  multiple of $N$ such that
  \begin{equation}  \label{def:discrep2}
  \omega_{k_*^\delta} \, = \, \omega_{k_*^\delta+1} \, = \, \dots \, = \,
  \omega_{k_*^\delta+N-1} \, = \, 0 .
  \end{equation}

  \noindent ii) The element \ $x_{k_*^\delta}^\delta$ \ satisfies \
  $\norm{ F_i( x_{k_*^\delta}^\delta ) - y_i^\delta } \ \leq \
  \tau \delta_i \, , \ i = 0, \dots, N$.
  \medskip

  \noindent iii) For $j < k_*^\delta$, there exists at least one index \
  $l \in \set{  \lfloor j \rfloor, \dots, \lfloor j \rfloor + N-1}$
  with $\omega_l \neq 0$. In other words, in the $\lfloor j \rfloor$th-cycle,
  for (at least) one of the $N$ equations in \eqref{eq:inv-probl} it holds \
  $\norm{ F_{[l]}( x_{l}^\delta ) - y_{[l]}^\delta } \ > \ \tau \delta_l$.
\end{remark}

Notice that, if $\| F_{[k],\delta}(x_{k}^\delta) \| > \tau \delta_{[k]}$ then
$\| F_{[k]}'(x_k^\delta)^* F_{[k],\delta}(x_k^\delta) \| \not = 0$.
This fact follows from Proposition~\ref{pr:prelim} item~\ref{it:3} (choose $\bar x = x^\star$
and $x = x_k^\delta$), since all $F_{i,\delta}$ also satisfy \textbf{A1} and \textbf{A2}.
Consequently, the sequence $(x_k^\delta)$ defined by iteration \eqref{eq:plwk-iter},
\eqref{def:wk-def} is well defined for $k = 0, \dots, k_*^\delta$.

The next result estimates the \emph{gain} in the square of the iteration error
$\norm{x^\star - x_k^\delta}$ for the PLWK method.

\begin{proposition}
  \label{pr:monot}
  Let assumptions \textbf{A1} -- \textbf{A3} hold true and $\theta_k \in (0,2)$. If
  $x_k^\delta \in B_\rho(x_0)$ and $\norm{F_{[k],\delta}(x_k^\delta)} > \tau \delta_{[k]}$,
  then
  \begin{align} \label{eq:plw.monE}
    \norm{x^\star - x_k^\delta}^2 \ \geq \ \norm{x^\star - x_{k+1}^\delta}^2 \ + \
    \theta_k (2-\theta_k) \, 
    \left( \dfrac{ p_{[k]}( \norm{F_{[k],\delta}(x_k^\delta)} ) }
    { \norm{F_{[k]}'(x_k^\delta)^* F_{[k],\delta}(x_k^\delta)} } \right)^2 ,
  \end{align}
  for all $x^\star \in B_\rho(x_0) \cap F_{[k]}^{-1}(y)$ and, in particular, for all
  $x^\star$ satisfying \textbf{A3}.
\end{proposition}

\begin{proof}
  If $x_k^\delta \in B_\rho(x_0)$ and $\| F_{[k],\delta}(x_{k}^\delta)
  \| > \tau \delta_{[k]}$, then $w_k=1$ and $x_{k+1}^\delta$ is a
  relaxed orthogonal projection of $x_k^\delta$ onto
  $H_{[k],x_k^\delta}$ with a relaxation factor $\theta_k$. The
  conclusion follows from this fact, the iteration formula
  \eqref{eq:plwk-iter}, and the separation Lemma~\ref{lm:separat}
  (compare with \cite[Prop.\,4.2]{LS15}).
\end{proof}

Proposition~\ref{pr:monot} is an essential tool for proving that $(x_k^\delta)$
does not leave the ball $B_\rho(x_0)$ for $k = 0, \dots, k_*^\delta$.
The next theorem guarantees this fact, as well as the finiteness of the stopping index
$k_*^\delta$ in the noisy data case (i.e., whenever $\min\set{\delta_0,\dots,\delta_{N-1}}
> 0$).

\begin{theorem}
  \label{th:plw-noise}
  If Assumptions \textbf{A1} -- \textbf{A3} hold true and $\theta_k \in (0,2)$,
  then the sequence $(x_k^\delta)$ in \eqref{eq:plwk-iter}, \eqref{def:wk-def}
  (with $\tau$, $p_i$, $\lambda_k$ as in \eqref{def:tau-p-lbd}) is well defined and
  \begin{align}  \label{eq:plw-noise}
    x_k^\delta \in B_{\rho/2}(x^\star)\subset B_\rho(x_0) \, ,
    \quad  k = 0, \dots, k^\delta_* \, ,
  \end{align}
  where $k^\delta_*$ is the stopping index defined in \eqref{def:discrep}.
  Moreover, if $\theta_k\in[a,b]\subset(0,2)$ for all $k\leq k^\delta_*$, then
  $k_*^\delta = O(\delta_{min}^{-2})$, where \
  $\delta_{min} := \min \{ \delta_0, \dots, \delta_{N-1} \}$.

  Additionally, in the particular case of exact data, the sequence $(x_k)$ defined
  by the PLWK method is well defined, 
  $x_k \in B_{\rho/2}(x^\star) \subset B_\rho(x_0)$ for all $k \in \N$,
  \begin{equation}  \label{eq:plw-noise1}
    \summ_{k=0}^\infty \, \lambda_k \, \norm{ F_{[k],0}(x_k) }^2 \ < \ \infty
  \end{equation}
  and
  \begin{equation}  \label{eq:plw-noise2}
    \summ_{k=0}^\infty \, \norm{ x_{k+1} - x_k }^2 \ < \ \infty .
  \end{equation}
\end{theorem}

\begin{proof}
  The proof of the first statement follows using an inductive argument. Indeed,
  $x_0^\delta = x_0$ obviously satisfies \eqref{eq:plw-noise}. Moreover,
  if $\| F_{[k],\delta}(x_{k}^\delta) \| \leq \tau \delta_{[k]}$ then $\omega_{k} = 0$
  and $x_{k+1}^\delta = x_{k}^\delta$.
  Otherwise, inequality \eqref{eq:plw.monE}, assumption $0 < \theta_k < 2$, and
  \textbf{A3} imply $x_{k+1}^\delta \in B_{\rho/2}(x^\star) \subset B_\rho(x_0)$.

  To prove the second statement, first observe that since $\theta_k\in[a,b]$, we
  have $\theta_k(2-\theta_k)\geq a(2-b)>0$. Thus,  it follows from Proposition
  \ref{pr:monot} that for any $k< k^\delta_*$
  \begin{multline}  \label{eq:plw-noise3}
    \|x^\star - x_0 \|^2 \ \geq \
    a (2-b) \!\!\!\!\!
    \sum_{ \substack{{j=0}\\{F_{[j],\delta}(x_j^\delta) \neq 0}} }^{k}
    \!\!\!\!\! \omega_j \left( \dfrac{p_{[j]}(\norm{F_{[j],\delta}(x_j^\delta)})}
    {\norm{F_{[j]}'(x_j^\delta)^* F_{[j],\delta}(x_j^\delta)}} \right)^2 \ \geq \\
    \geq \
    \dfrac{a(2-b)}{C^2} \!\!\!\!\!
    \sum_{ \substack{{j=0}\\{F_{[j],\delta}(x_j^\delta) \neq 0}} }^{k}
    \!\!\!\!\! \omega_j \left( \dfrac{p_{[j]}(\norm{F_{[j],\delta}(x_j^\delta)})}
    {\norm{F_{[j],\delta}(x_j^\delta)}} \right)^2 \! .
  \end{multline}
  Observe that, if $t > \tau \delta_i$, then
  \begin{align*}
    \dfrac{p_i(t)}{t} = (1-\eta) t - (1+\eta) \delta_i >
    \left[ \tau - \dfrac{1+\eta}{1-\eta} \right] (1-\eta) \delta_i >
    \widetilde{C} \, \delta_{min} ,
  \end{align*}
  where $\widetilde{C} := [(1-\eta)\tau - (1+\eta)]$.
  On the other hand, as already observed in Remark~\ref{rem:stop.index}, item~(iii),
  each cycle $l_0$ with $0 \leq l_0 < \lfloor k^\delta_* / N \rfloor$ contains at
  least one index $l = l_0.N + l_1$ (with $l_1 \in \{0, \dots, N-1\}$) such that
  $\| F_{[l],\delta}(x_{l}^\delta) \| = \| F_{l_1,\delta}(x_{l}^\delta) \| >
  \tau \delta_{l_1} = \tau \delta_{[l]}$, i.e., $w_l = 1$.
  Therefore, for any $k < k^\delta_*$
  $$
  \|x^\star - x_0^\delta \|^2 \ \geq \
  \dfrac{a(2-b)}{C^2} \, \widetilde{C}^2 \, \delta_{min}^2 \, \lfloor k/N \rfloor ,
  $$
  from were we conclude $k_*^\delta = O(\delta_{min}^{-2})$.
  
  Next we address the statements related to the exact data case. Arguing
  as in the first part of the proof, one concludes that the sequence
  $(x_k)$ is well defined and satisfies $x_k \in B_{\rho/2}(x^\star)
  \subset B_\rho(x_0)$, for all $k \geq 0$. In order to prove
  \eqref{eq:plw-noise1}, notice that if the data is exact then \ $p_i(t)
  = (1-\eta)\, t^2$ for $i = 0, \dots, N-1$. Thus, it follows from
  \eqref{eq:plw-noise3} that
  \begin{multline*}
    \|x^\star - x_0 \|^2 \ \geq \
    a (2-b) \!\!\!\!\!
    \sum_{ \substack{{j=0}\\{F_{[j],\delta}(x_j^\delta) \neq 0}} }^{k}
    \!\!\!\!\!  \left( \dfrac{p_{[j]}(\norm{F_{[j],0}(x_j)})}
    {\norm{F_{[j]}'(x_j)^* F_{[j],0}(x_j)}} \right)^2  \geq \\
    \geq \ 
    a(2-b) \!\!\!\!\!
    \sum_{ \substack{{j=0}\\{F_{[j],\delta}(x_j^\delta) \neq 0}} }^{k}
    \!\!\!\!\! (1-\eta) \lambda_j \norm{ F_{[j],0}(x_j) }^2
    \ = \
    a(2-b)(1-\eta) \sum_{j=0}^k \, \lambda_j \norm{ F_{[j],0}(x_j) }^2 ,
  \end{multline*}
  for all $k \in \N$ (the identity follows from \eqref{def:wk-def} and
  \eqref{def:tau-p-lbd}), proving \eqref{eq:plw-noise1}.
  Finally, in order to prove \eqref{eq:plw-noise2} we derive from
  \eqref{eq:plw-iter}, \eqref{def:wk-def} and \eqref{def:tau-p-lbd} the estimate
  \begin{align*}
    \|x_{k+1} - x_k \|^2 \ &  =  \
    \theta_k^2 \, \omega_k^2 \, \lambda_k^2 \, \norm{ F_{[k]}'(x_k)^* F_{[k],0}(x_k) }^2 \\
    & \leq \ 4 \, \lambda_k^2 \, \norm{ F_{[k]}'(x_k)^* F_{[k],0}(x_k) }^2
    \  =  \  4 \, \lambda_k \, \norm{ F_{[k],0}(x_k) }^2 .
    \end{align*}
  Therefore, \eqref{eq:plw-noise2} follows from \eqref{eq:plw-noise1}.
\end{proof}

\section{Convergence analysis}
\label{sec:conv-anal}

We start by stating and proving a convergence result for the PLWK method in the
case of exact data.
Theorem~\ref{th:lw.lim} gives a sufficient condition for weak convergence of
the relaxed PLWK iteration to some element $\bar x \in B_\rho(x_0)$, which is
a solution of \eqref{eq:inv-probl}.

In the Appendix an alternative strong convergence result for the PLWK method is
given (see Theorem~\ref{th:lw.lim.s}). The proof of this result, however, requires
a modification in the definition of the stepsize $\lambda_k$ in \eqref{def:tau-p-lbd}
(for details, please see \eqref{def:trunc-lbd} below).

\begin{theorem}[Convergence for exact data] \mbox{} \\
  \label{th:lw.lim}
  Let assumptions \textbf{A1} -- \textbf{A3} hold true, $\delta_0 = \dots = \delta_{N-1}
  = 0$ and $(x_k)$ be defined by the PLWK method in \eqref{eq:plwk-iter},
  \eqref{def:wk-def} with $\tau$, $p_i$, $\lambda_k$ as in \eqref{def:tau-p-lbd}.
  If \ $\inf$ $\theta_k > 0$ \ and \ $\sup$ $\theta_k < 2$, then $(x_k)$ converges
  weakly to some $\bar x \in B_\rho(x_0)$ solving \eqref{eq:inv-probl}.
\end{theorem}

\begin{proof}
The proof is divided in four main steps:
\medskip

\noindent \textbf{(i)} $\norm{ F_{[k]}(x_{k}) - y_{[k]} }
   \to 0$ as $k \to \infty$. \\
Let $Q \subset \N$ be the set of indices $k$ such that $\lambda_k\neq 0$.
Then, it follows from \eqref{eq:plw-noise1} that%
\footnote{Notice that, for exact data $\lambda_{k} = 0$ iff $F_{[k],0}(x_{k}) = 0$.}
\begin{eqnarray*}
\infty & > & \summ_{k \in Q} \, \lambda_{k} \, \norm{ F_{[k],0}(x_{k}) }^2 \\
       & = & (1-\eta) \summ_{k \in Q} \, \norm{ F_{[k],0}(x_{k}) }^4
                      \norm{ F'_{[k]}(x_{k})^* F_{[k],0}(x_{k}) }^{-2} \\
    & \geq & (1-\eta) C^{-2} \summ_{k\in Q} \, \norm{ F_{[k],0}(x_{k}) }^2
     = (1-\eta) C^{-2} \summ_{k \in \N} \, \norm{ F_{[k],0}(x_{k}) }^2  .
\end{eqnarray*}
To complete the proof of this first step, we use the above inequalities and recall that
$F_{i,0}(x)=F_i(x)-y_i$.
\medskip

\noindent \textbf{(ii)} Every $\bar x$ weak limit of a subsequence of $(x_k)$ satisfy the
equations $F_i(\bar x)=y_i$.\\
Suppose that $x_{k_j} \rightharpoonup \bar{x}$. Take $i\in\set{0,\dots,N-1}$.
In view of the definition of $[k]$, for each $j$ there exists a $k'_j$ such that
\[
[ k'_j ] = i, \quad k_j \leq k'_j \leq k_j+N-1 .
\]
Since
\[
\norm{x_{k_j}-x_{k'_j}}\leq \sum_{k=k_j}^{k_j+N-2}\norm{x_{k+1}-x_k},
\]
it follows from~\eqref{eq:plw-noise2} that $x_{k'_j} \rightharpoonup
\bar{x}$. It follows from step \textbf{(i)} and the definition of $k'_j$
that that $F_{i,0}(x_{k'_j})\to 0$. Since $F_i$ satisfies the TCC, it
follows from Proposition~\ref{pr:tcc-cl-gr} that $F_i(\bar x)-y_i=0$. \medskip

\noindent
\textbf{(iii)} The sequence $(x_k)$ has a unique weak adherent point $\bar x$ and
such a point belongs to the set $B_\rho(x_0)$. \\
Since the data is exact, Theorem~\ref{th:plw-noise} guarantees that $(x_k)$ is in 
$B_{\rho/2}(x_0$). Hence, there exists a subsequence $(x_{k_j})$ converging weakly
to some $\bar x\in B_\rho(x_0)$. Suppose that $(x_{m_j})$ converges to 
$\hat x$. By step \textbf{(ii)}, $F_i(\bar x)=y_i=F_i(\hat x)$ for $i=\set{0,\dots,N-1}$.
It follows from this result and Proposition~\ref{pr:monot} that
\begin{align*}
  \norm{\bar x-x_{k+1}}\leq\norm{\bar x-x_k} \, , \quad
  \norm{\hat x-x_{k+1}}\leq\norm{\hat x-x_k} \, , \qquad k=1,2,\dots
\end{align*}
If $\hat x \neq \bar x$, it follows from the above inequalities and Opial's
Lemma~\cite{Opi67} that
{\color{black}
$$
\lim_{k\to\infty} \norm{\bar x-x_k} 
  = \lim\inf_{j\to\infty} \norm{\bar x - x_{k_j}}
  < \lim\inf_{j\to\infty} \norm{\hat x - x_{k_j}}
  = \lim_{k\to\infty} \norm{\hat x - x_k}
$$
}
and
{\color{black}
$$
\lim_{k\to\infty} \norm{\hat x-x_k} 
  = \lim\inf_{j\to\infty} \norm{\hat x-x_{m_j}}
  < \lim\inf_{j\to\infty} \norm{\bar x-x_{m_j}}
  = \lim_{k\to\infty} \norm{\bar x-x_k} \, ,
$$
}
which is an absurd. 
\medskip

\noindent {\bf (iv)} The sequence $(x_k)$ converges weakly to $\bar{x}$. \\
Since the $(x_k) \in B_\rho(x_0)$ is a bounded sequence, this assertion follows from
step \textbf{(iii)}.
\end{proof}

In the next theorem we discuss a stability result, which is an essential
tool to prove the last result of this section, namely Theorem~\ref{th:lw.scE}
(the semi-convergence of the PLW method). Notice that this is the first time
were the strong assumption \textbf{A4} is needed in this manuscript.

\begin{theorem}
  \label{th:lw.stab}
  Let assumptions \textbf{A1} -- \textbf{A4} hold true. For each fixed $k \in \N$,
  the element $x_k^\delta$, computed after kth-iterations of the PLWK method 
  \eqref{eq:plwk-iter}, depends continuously on the data $y_i^\delta$.
\end{theorem}
\begin{proof}
From \eqref{def:tau-p-lbd}, \ assumptions \textbf{A1}, \textbf{A4} \ and
Theorem~\ref{th:plw-noise}, it follows that the mappings
$\varphi_i: D(\varphi_i) \to X$ with
\begin{align*}
& D(\varphi_i) := \big\{ (x, y_i^\delta, \delta_i) \ | \ x \in D ;
                  \ \delta_i > 0 ;
                  \ \norm{y_i^\delta - y_i} \leq \delta_i ;
                  \ F_i'(x)^* (F_i(x) - y_i^\delta) \neq 0 \big\} , \\
& \varphi_i(x,y_i^\delta,\delta_i) := x - \dfrac{p_i(\norm{F_i(x)-y_i^\delta})}
                                    {\norm{F_i'(x)^* (F_i(x)-y_i^\delta)}^2}
                                    \, F_i'(x)^* (F_i(x) - y_i^\delta)
\end{align*}
are continuous on the corresponding domains of definition. Therefore, whenever the
iterate \
$x_k^\delta = \big( \varphi_{[k]}(\cdot, y^\delta_{[k]}, \delta_{[k]}) \big)
\circ \cdots \circ \big( \varphi_{0}(\cdot, y^\delta_{0}, \delta_{0}) \big) (x_0)$
is well defined,%
\footnote{This composition is to be understood in a cyclic way.}
it depends continuously on $(y_i^\delta, \delta_i)_{i=0}^{N-1}$.
\end{proof}

Theorem~\ref{th:lw.stab} together with Theorem~\ref{th:lw.lim} are the key ingredients
in the proof of the next result, which guarantees that the stopping rule 
\eqref{def:discrep} renders the PLWK iteration a regularization method.
The proof of Theorem~\ref{th:lw.scE} uses classical techniques from the analysis of
Landweber-type iterative regularization techniques (see, e.g., \cite[Theor.~11.5]{EHN96}
or \cite[Theor.~2.6]{KNS08}) and thus is omitted.

\begin{theorem}[semi-convergence]
  \label{th:lw.scE}
  Let assumptions \textbf{A1} -- \textbf{A4} hold true,
  $(\delta_0^j, \dots,\delta_{N-1}^j)_j \to 0$ as $j\to\infty$,
  and $(y_0^j, \dots, y_{N-1}^j ) \in Y^N$ be given with
  $\norm{y_i^j - y_i} \leq \delta_i^j$ for $i \in \set{0, \dots, N-1}$ and $j\in\N$.
  If the PLWK iteration \eqref{eq:plwk-iter} is stopped with $k_*^j$ according to
  \eqref{def:discrep}, then $(x_{k_*^j}^\delta)$ converges weakly to a solution
  $\bar x \in B_\rho(x_0)$ of \eqref{eq:inv-probl} as $j\to\infty$.
\end{theorem}

\section{The randomized PLWK method}
\label{sec:plwkr}

{\color{black}
In the spirit of~\cite{BB97}, we consider a ``randomized'' version of the PLWK method
}
where in the $q$-th cycle $k=(q-1)N,(q-1)N+1,\dots,qN-1$,
\begin{align*}
  [(q-1)N], \ [(q-1)N+1], \ \dots, \ [qN-1]
\end{align*}
is a random permutation of $0,\dots,N-1$. 
In our numerical tests, the randomized version of the PLW method performed
slightly better than the deterministic version.

All convergence results stated for the ``deterministic'' PLWK method extend
trivially for the ``randomized version'' (here called PLWKr), provided the
same sequence of random permutations is considered in Theorems~\ref{th:lw.stab}
and \ref{th:lw.scE}.

\section{Numerical experiments} \label{sec:numerics}

In this section the PLWK method is implemented for solving an exponentially ill-posed
inverse problem related to the Dirichlet to Neumann map and its performance is
compared against the benchmark methods LWK (Landweber-Kaczmarz \cite{HLS07,HKLS07})
and LWKls (Landweber-Kaczmarz with line search \cite{CHLS08}).

\subsection{The inverse doping problem}

We briefly describe the inverse doping problem considered in \cite{Le06, LMZ06a, LS15}
with the same setup used in \cite[Section~5.3]{LS15}.
This problem consists in determining the doping profile function from measurements
of the linearized Voltage-Current map.

After several simplifications, the problem becomes to identify the parameter
function $\gamma$ in the PDE model
\begin{equation} \label{eq:pde-d2n}
-{\rm div}\, (\gamma \nabla \hat u) \ = \ 0 \,\ {\rm in} \ \Omega \quad\quad
\hat u \ = \ U(x) \,\ {\rm on} \ \partial\Omega 
\end{equation}
from measurements of the Dirichlet-to-Neumann map
%
$$
  \Lambda_\gamma : \begin{array}[t]{rcl}
    H^{1/2}(\partial\Omega) & \to & H^{-1/2}(\partial\Omega) \, , \\
    U & \mapsto & \big( \gamma^\star \hat{u}_\nu \big) |_{\partial\Omega}
  \end{array}
$$
where $\gamma^\star$ is the exact coefficient to be determined.
Only a \emph{finite} number $N$ of 
measurements is available, i.e., one knows 
$$
\big\{ (U_i, \Lambda_{\gamma^\star}(U_i)) \big\}_{i=0}^{N-1}
\ \in \ \big[ H^{1/2}(\partial\Omega) \times H^{-1/2}(\partial\Omega) \big]^N .
$$
Moreover, $\gamma^\star$ is assumed to be known at $\partial\Omega$, the boundary of the
domain $\Omega \subset \R^2$ representing the semi-conductor device \cite{BELM04}.


In \cite[Section~5.3]{LS15} this inverse problem was addressed for $N = 1$ (i.e.,
parameter identification from a single experiment). Here the more general setting
$N \geq 1$ is considered, which can be written within the abstract framework of
\eqref{eq:inv-probl} with
\begin{equation} \label{eq:ip-d2n}
 F_i(\gamma) \ = \ \Lambda_\gamma(U_i),\quad  \ y_i=\Lambda_{\gamma^\star}(U_i) \, , \quad
 i = 0, \dots, N-1 \, ,
\end{equation}
where $U_i \in H^{1/2}(\partial\Omega)$ are fixed Dirichlet boundary conditions
(representing the voltage profiles for the experiments), 
$Y := H^{1/2}(\partial\Omega)$
and $X := L^2(\Omega) \supset D_i := \{ \gamma \in L^\infty(\Omega)$;
$0 < \gamma_m \le \gamma(x) \le \gamma_M$, a.e. in $\Omega \}$.

The operators $F_i: H^1(\Omega) \ni \gamma \mapsto\Lambda_\gamma(U_i) \in H^{-1/2}(\partial\Omega)$
in \eqref{eq:ip-d2n} are continuous maps \cite{BELM04}.
Up to now, it is not known whether the $F_i$'s satisfy the TCC \eqref{eq:tcc}.
However, in \cite{LR08} it was established that the discretization of each $F_i$
in \eqref{eq:ip-d2n}, using the finite element method, does satisfy the TCC.
Furthermore, for each fixed $U = U_i$ in \eqref{eq:pde-d2n}, the map
$H^1(\Omega) \ni \gamma \mapsto \hat{u} \in H^1(\Omega)$ satisfies the
TCC with respect to the $H^1(\Omega)$ norm \cite{KNS08}.
Due to these considerations, the analytical convergence results of
Sections~\ref{sec:plwk} and~\ref{sec:conv-anal} do apply to finite-element
discretizations of \eqref{eq:ip-d2n} in this particular setting.
Moreover, $H^1(\Omega)$ is a natural choice of parameter space for the
PLW and PLWK methods.

\subsection{Setup of the numerical experiments}

The setup of the numerical experiments presented in this section is as follows:

\smallskip\noindent $\bullet$ The domain $\Omega \subset \mathbb R^2$ for the elliptic
PDE model \eqref{eq:pde-d2n} is the unit square $(0,1) \times (0,1)$ and the parameter
space for the above described inverse problem is $H^1(\Omega)$.

\smallskip\noindent $\bullet$ The ``exact solution'' $\gamma^\star \in D_i \subset H^1(\Omega)$
of system \eqref{eq:ip-d2n} is shown in Figure~\ref{fig:calderon-setup} (Top).

\smallskip\noindent $\bullet$ The number of available experiments is $N = 12$ and the
Dirichlet boundary conditions used in \eqref{eq:pde-d2n} are the continuous functions
$U_i: \partial\Omega \to \R$, $i = 0, \dots, N-1$, defined by
\begin{align}
  U_{2i}=\sin(s(t)(i+1)\pi/2),\quad
 U_{2i+1}=\cos(s(t)(i+1)\pi/2)
\end{align}
where $s(t)$ is the length of the counterclockwise oriented arc along $\partial \Omega$,
connecting $(0,0)$ to $t$, that is
\begin{align*}
  s(t)=
  \begin{cases}
    x, & t=(x,0),\; 0\leq x<1\\
    1+y,& t=(1,y),\;0\leq y<1\\
    3-x,& t=(x,1),\; 0<x\leq 1\\
    4-y,& t=(0,y),\; 0<y\leq 1
  \end{cases}
\end{align*}
In Figure~\ref{fig:calderon-setup} (Center) two distinct voltage profiles $U_i(x)$ are
plotted, together with the corresponding solutions of (\ref{eq:pde-d2n}).

\smallskip\noindent $\bullet$ The TCC constant $\eta$ in \eqref{eq:tcc} is not known for
this particular setup. In our computations we used the value $\eta = 0.45$, which
is in agreement with assumption \textbf{A2} as well as with \cite[Eq.\,(1.5)]{HNS95}.

\smallskip\noindent $\bullet$ The ``exact data'' $y_i$ in \eqref{eq:ip-d2n} is
obtained by solving the direct problem \eqref{eq:pde-d2n} (with $\gamma = \gamma^\star$
and $U = U_i$) using a finite element type method and adaptive mesh refinement
(mesh with approx 131.000 elements).
%
%
In order to avoid inverse crimes, a coarser uniform mesh (with ca.~33.000 elements)
was used in the implementation of the finite element method, employed for
solving the PDE's related to the iterative methods tested.

\smallskip\noindent $\bullet$ The choice of the initial guess $\gamma_0$
is a critical issue. According to assumptions \textbf{A1} --
\textbf{A3}, $\gamma_0$ has to be sufficiently close to $\gamma^\star$,
otherwise the convergence analysis developed previously does not apply.
As explained in \cite[Remark~5.1]{LS15} we choose $\gamma_0$ as the
solution the Dirichlet boundary value problem \ $\Delta \gamma_0 = 0$ in
$\Omega$, \ $\gamma_0 = \gamma^\star$ at $\partial\Omega$.

\smallskip\noindent $\bullet$ In the numerical experiment with noisy data, artificially
generated (random) noise of 2\% was added to the exact data $y_i$ in order to generate
the noisy data $y_i^\delta$. For the verification of the stopping rule \eqref{def:discrep}
we assumed exact knowledge of the noise level and chose $\tau = 3$ in \eqref{def:tau-p-lbd},
which is in agreement with the above choice for $\eta$.

\smallskip\noindent $\bullet$ 
The computation of the
adjoints $F'_{i,\delta}(\gamma)^*$, for $i=0,\dots,N-1$,
is done using the $H^1$-inner product, as developed in \cite[Remark~5.2]{LS15}.

\subsection{Experiments for exact data and noisy data}

In our numerical experiments, we implement four different Landweber-Kaczmarz
type  methods for solving the ill-posed system \eqref{eq:ip-d2n}, namely,
\begin{description}
\item[LWK] Landweber-Kaczmarz method \cite{HLS07,HKLS07};
\item[LWKls] Landweber-Kaczmarz method with line-search \cite{CHLS08};
\item[PLWK] Projective Landweber-Kaczmarz method, as developed in
  Section~\ref{sec:plwk};
\item[PLWKr] randomized Projective Landweber-Kaczmarz method, as developed in
  Section~\ref{sec:plwkr}; 
\end{description}

In order to compare the performance of these methods, the iteration error as well
as the residual are computed at the end of each cycle, i.e., our plots describe
the quantities
$$
\norm{\gamma_{kN} - \gamma^\star}_{H^1(\Omega)}
\quad\quad {\rm and } \quad\quad
\summ_{i=0}^{N-1} \norm{F_i(\gamma_{kN}) - y_i}_{L^2(\partial\Omega)} \, ,
\quad k = 0, 1, 2, \dots
$$
(here $k$ is an index for cycles).

For solving the elliptic PDE's, needed for the implementation of the iterative
methods, we used the package PLTMG \cite{Ban94} compiled with GFORTRAN-4.8 in
a INTEL(R) Xeon(R) CPU E5-1650 v3.
\medskip

Evolution of iteration error and evolution of residual in the {\em exact data case}
are shown in Figure~\ref{fig:calderon-exact}.
The PLWK method (GREEN) is compared with the LWK method (BLUE), with the
LWK method using line-search (LWKls, RED) and with the randomized PLWK method
(PLWKr, LIGHT-BLUE).
\medskip

Evolution of iteration error and evolution of residual in the {\em noisy data case}
are shown in Figure~\ref{fig:calderon-noisy}.
The PLWK method (GREEN) is compared with the LWK method (BLUE), with the
LWK method using line-search (LWKls, RED) and with the randomized PLWK method
(PLWKr, LIGHT-BLUE).
The stop criteria \eqref{def:discrep} is reached after 29 steps for the PLWK iteration,
42 steps for the LWKls iteration, 22 steps for the PLWKr iteration, and 74 steps for
the LWK iteration.
\medskip

\begin{figure}[t!]
\centerline{ \includegraphics[width=10.0cm]{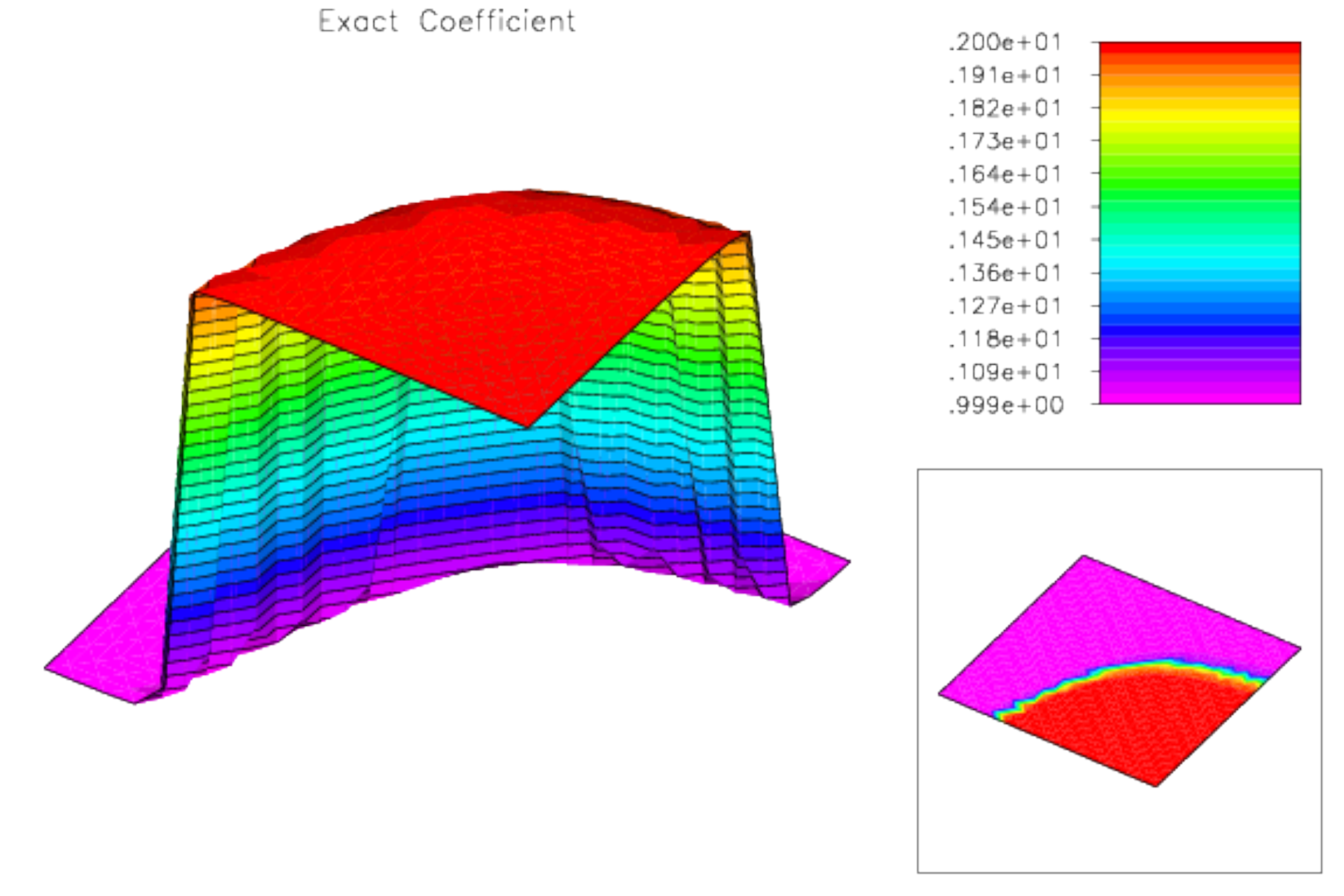} }
\bigskip\bigskip\bigskip
\centerline{ \includegraphics[width= 8.0cm]{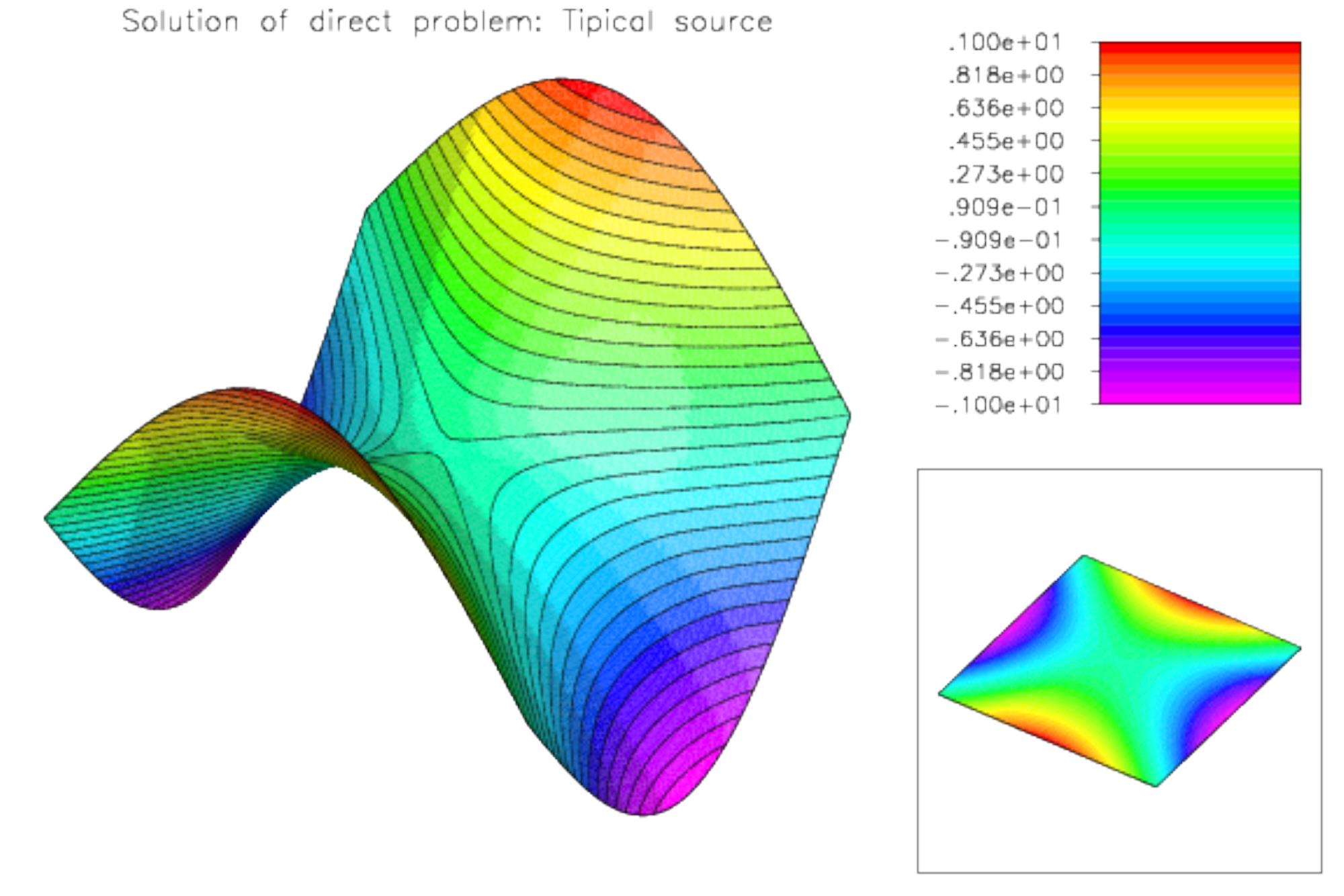}
             \includegraphics[width= 8.0cm]{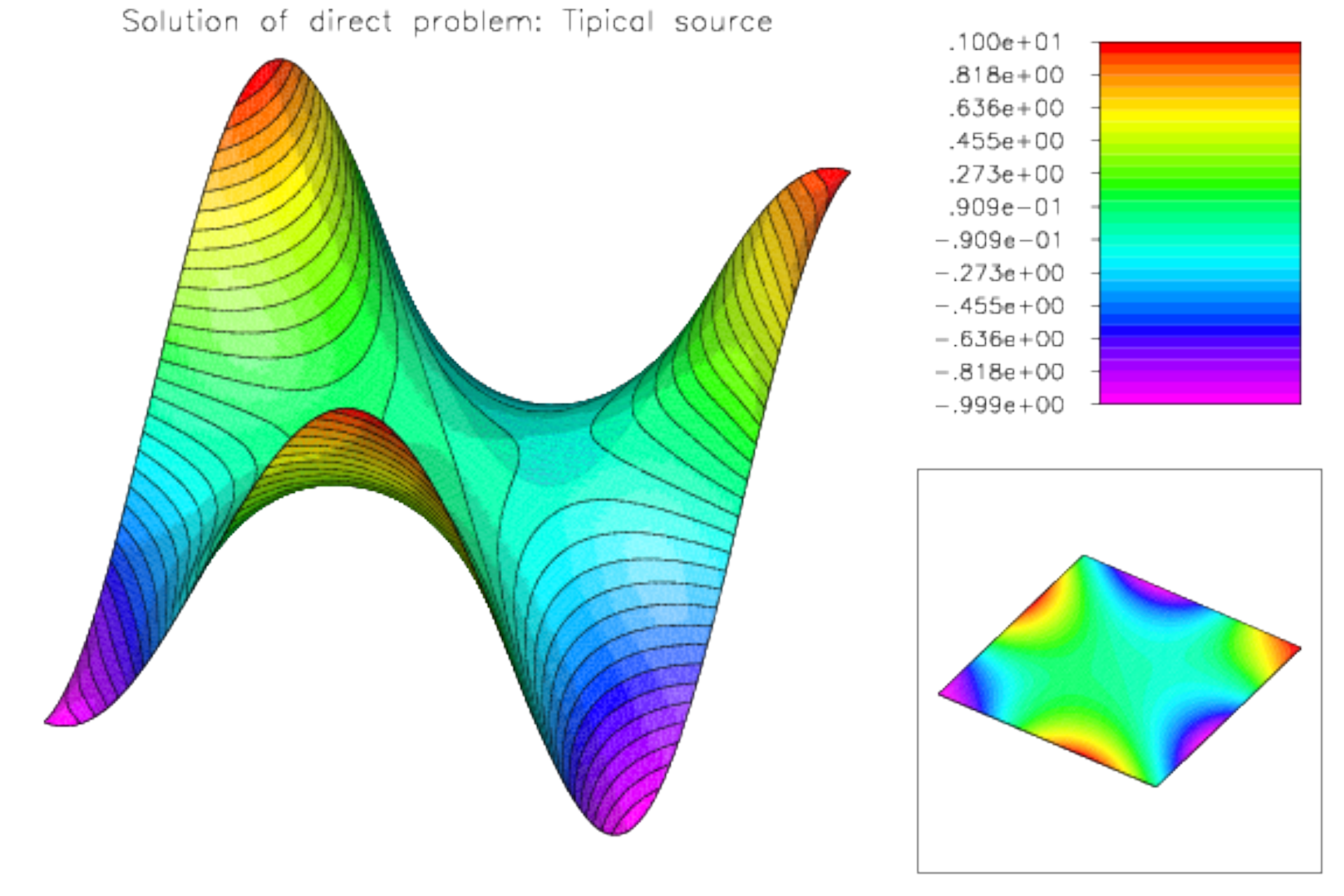} }
\bigskip\bigskip\bigskip
\centerline{ \includegraphics[width=10.0cm]{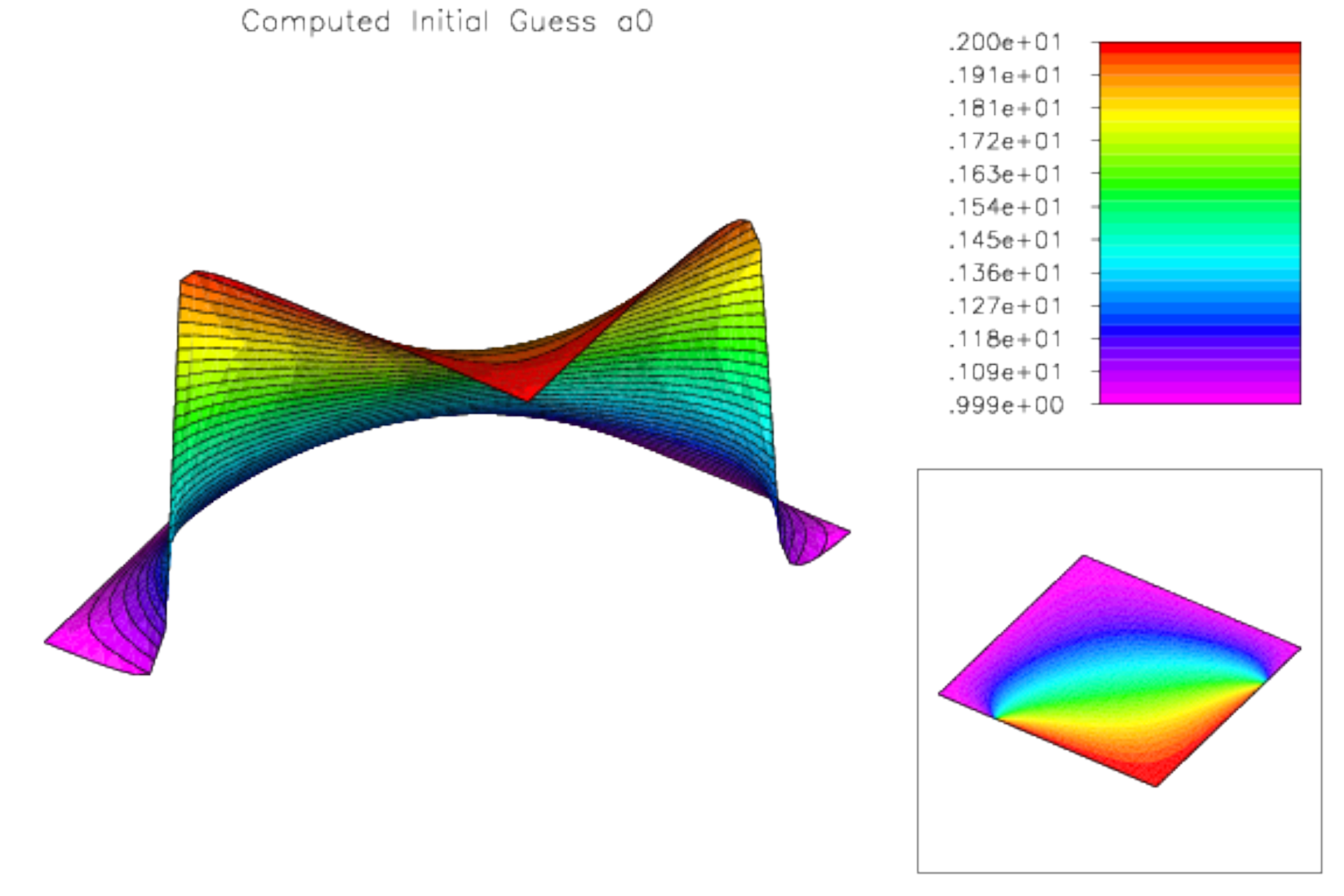} }
\caption{Setup of the inverse doping problem.
{\bf Top:} Parameter function $\gamma^*$ to be identified;
{\bf Center:} Functions $U_1$ and $U_6$ (the Dirichlet boundary conditions at 
$\partial\Omega$ for \eqref{eq:pde-d2n}) and the solutions $\hat u_2$, $\hat u_6$
of the corresponding PDE's;
{\bf Bottom:} Initial guess $\gamma_0$ for the iterative methods PLWK, LWK and LWKls.}
\label{fig:calderon-setup}
\end{figure}

\begin{figure}[t!]
\centerline{ \includegraphics[width=14cm]{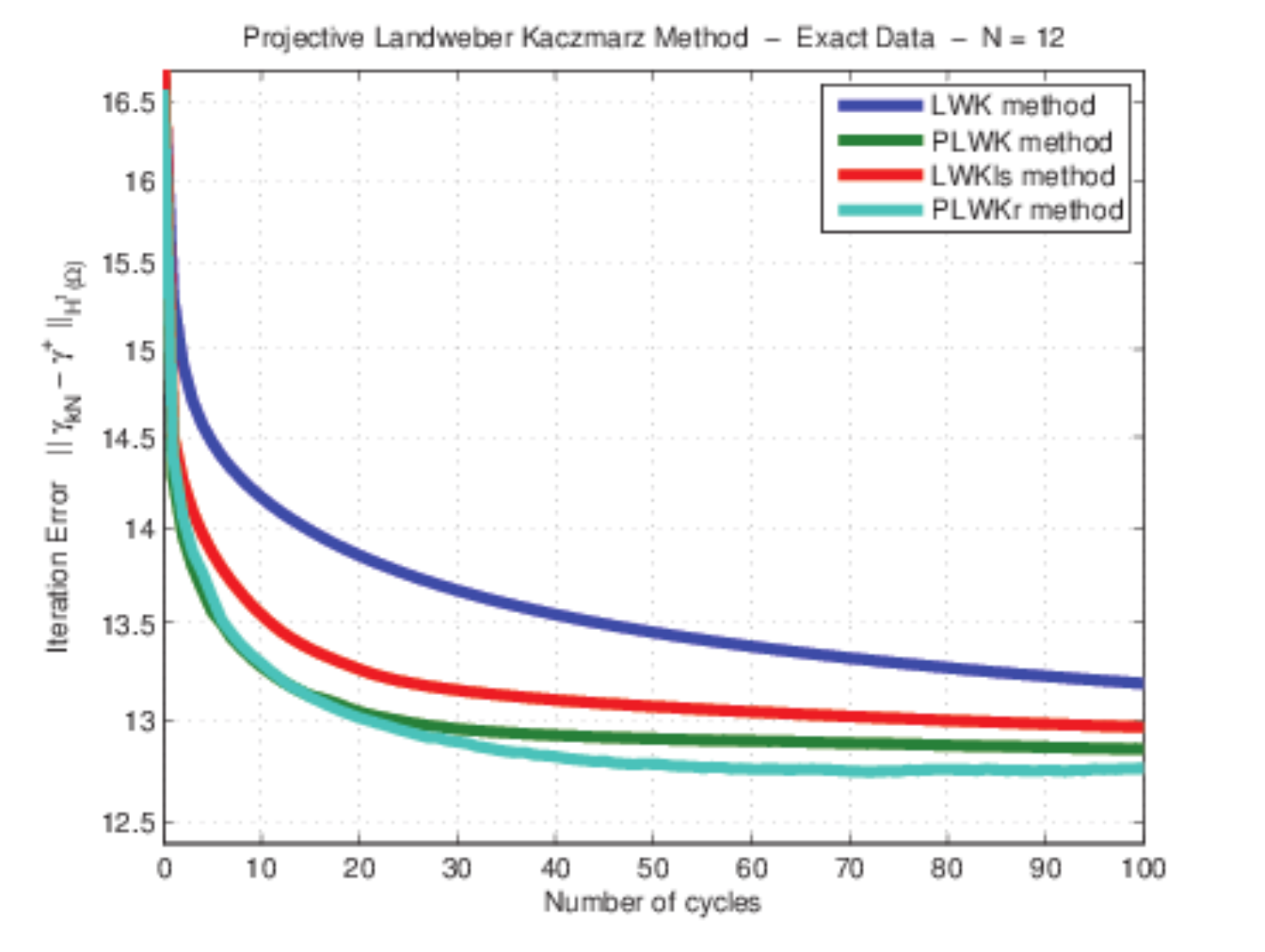} }
\centerline{ \includegraphics[width=14cm]{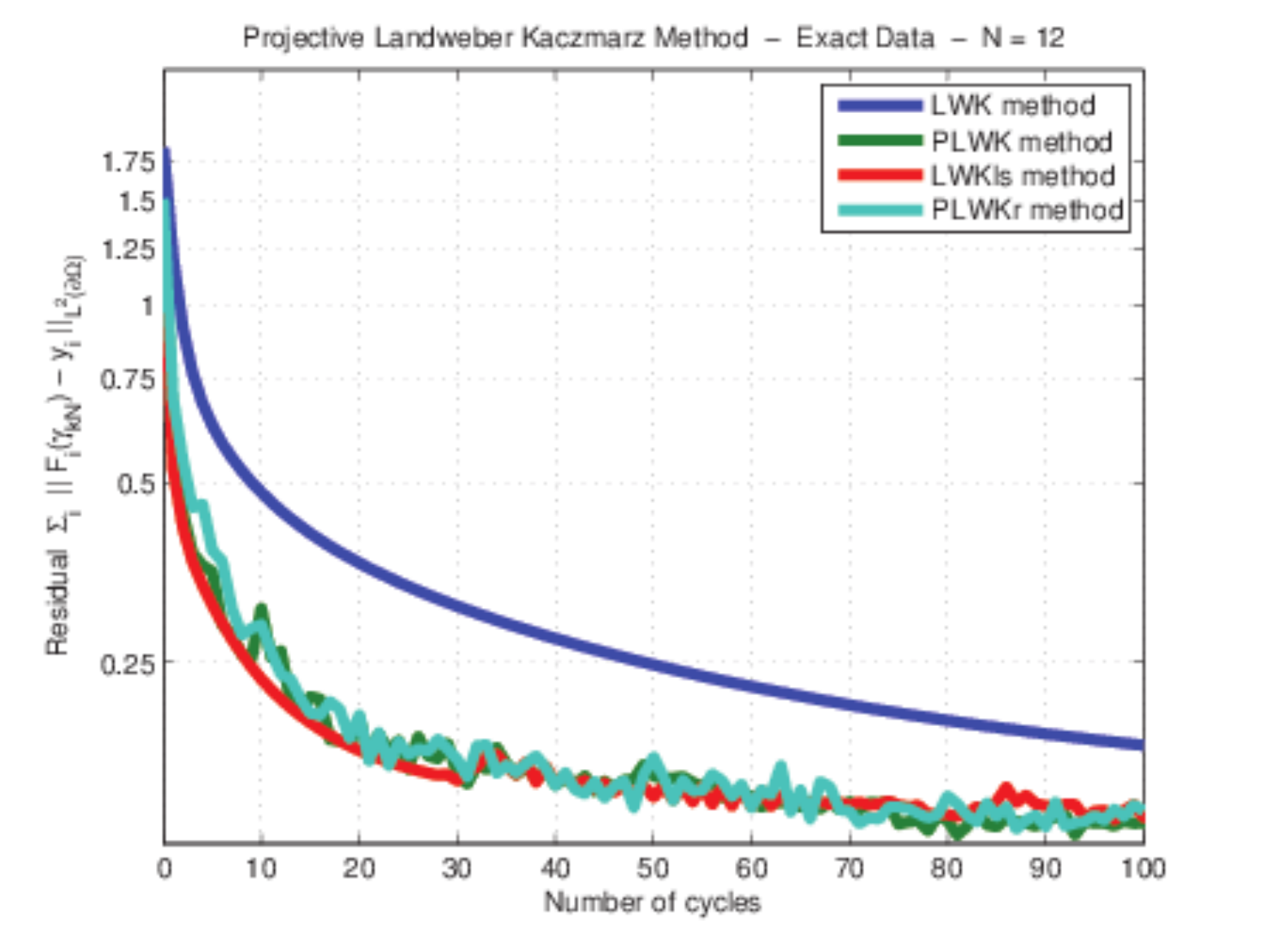} }
\caption{Experiment with exact data.
The PLW method (GREEN) is compared with the LW method (BLUE) and with the LWls method (RED).
{\bf Top:} Evolution of the iteration error $\norm{\gamma_{kN} - \gamma^\star}_{H^1(\Omega)}$;
{\bf Bottom:} Evolution of the residual
$\summ\nolimits_{i=0}^{N-1} \norm{F_i(\gamma_{kN}) - y_i}_{L^2(\partial\Omega)}$.}
\label{fig:calderon-exact}
\end{figure}

\begin{figure}[t!]
\centerline{ \includegraphics[width=10cm]{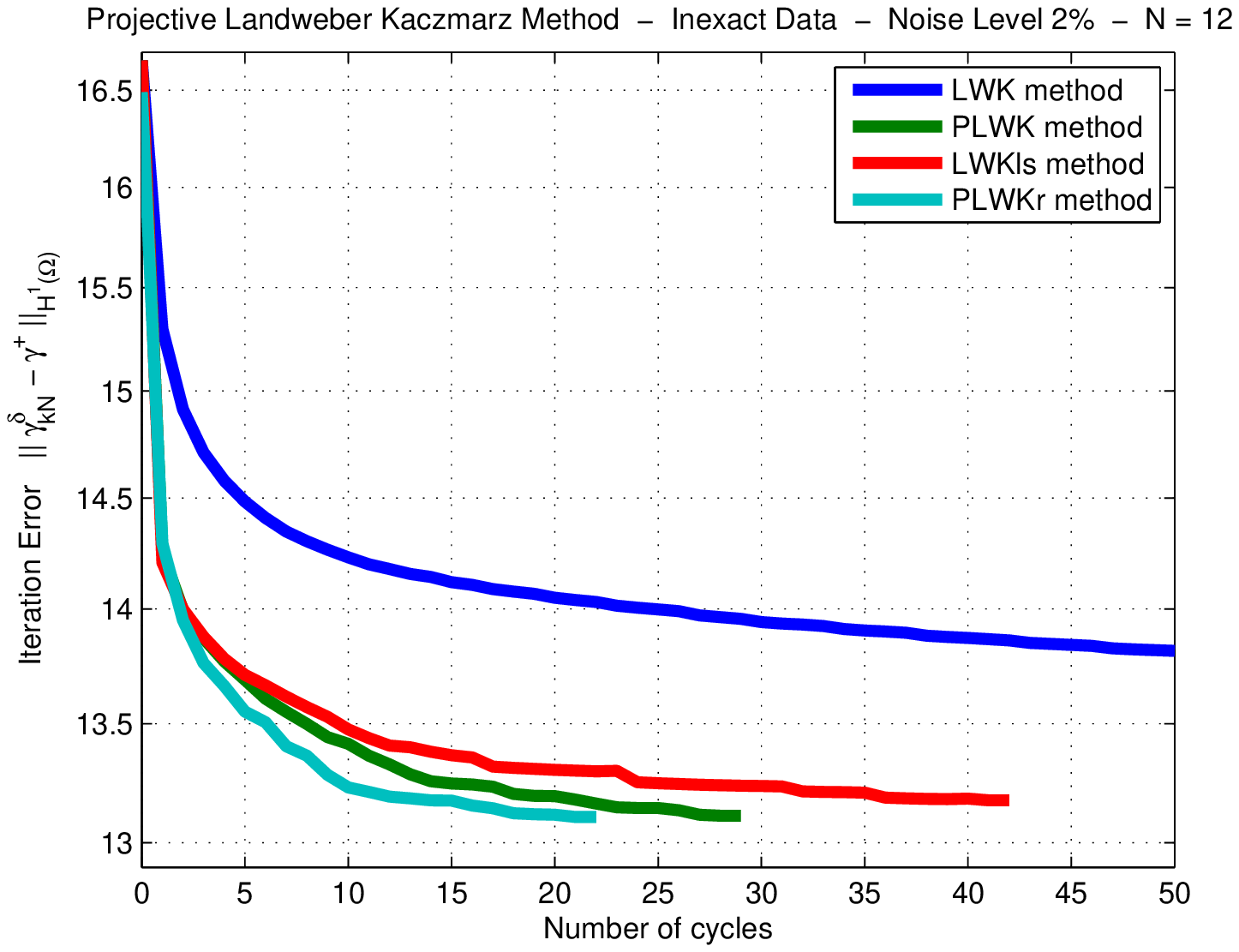} }
\centerline{ \includegraphics[width=10cm]{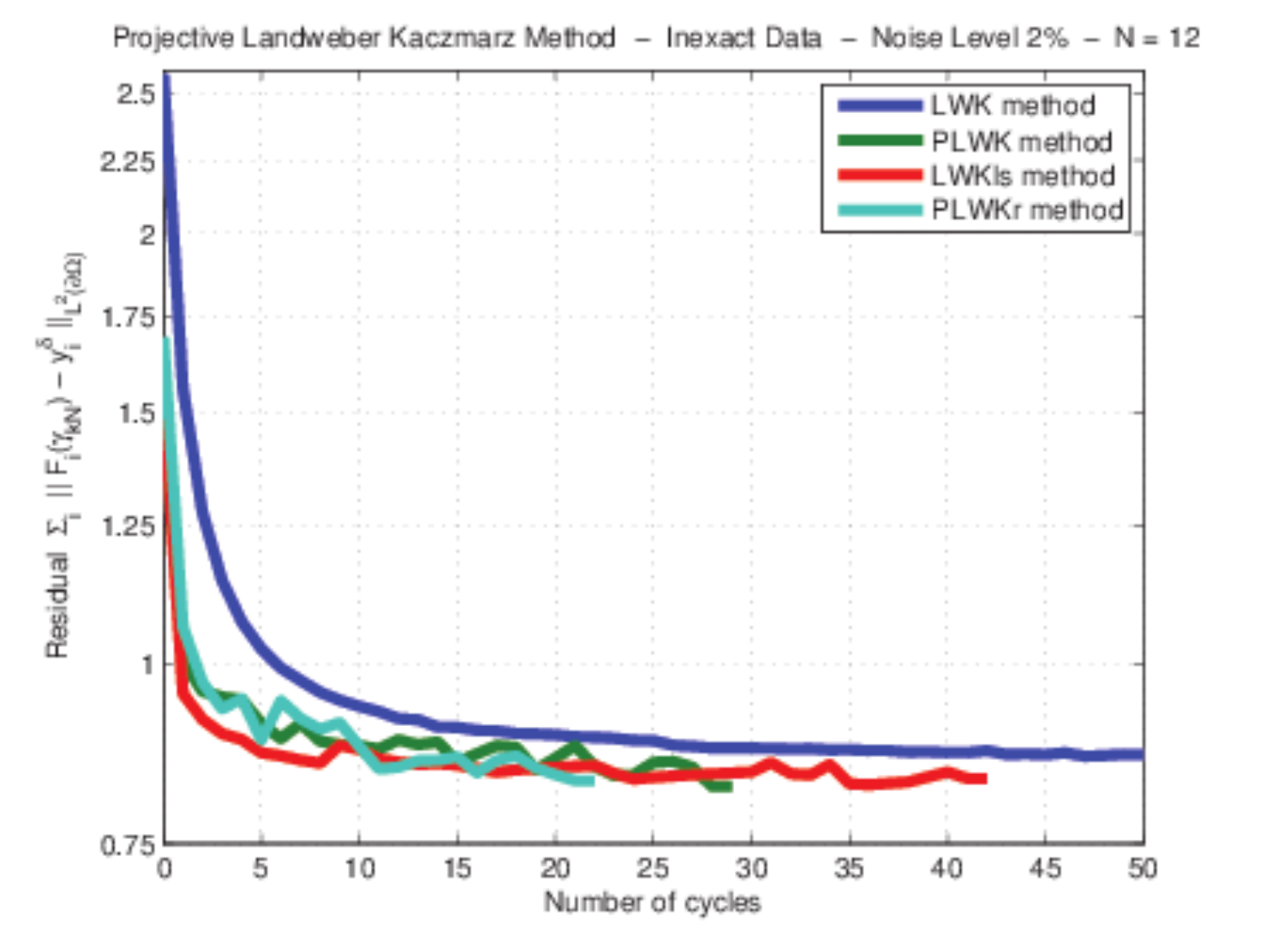} }
\centerline{ \includegraphics[width=9cm]{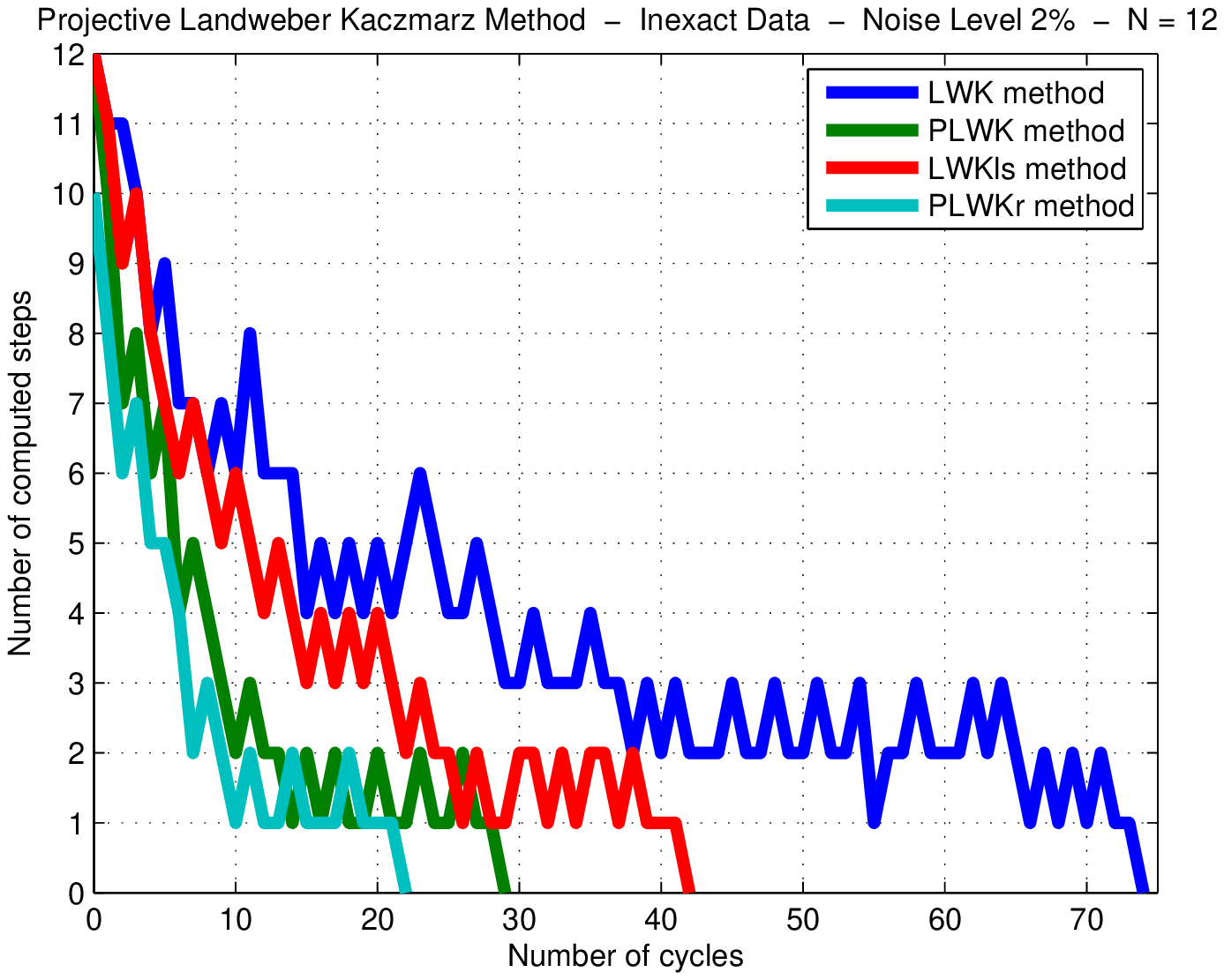} }
\caption{Experiment with noisy data.
The PLW method (GREEN) is compared with the LW method (BLUE), the LWls method (RED) and
the PLW-random method (LIGHT-BLUE).
{\bf Top:} Evolution of the iteration error;
{\bf Center:} Evolution of the residual
{\bf Bottom:} Number of computed iterative steps per cycle.}
\label{fig:calderon-noisy}
\end{figure}

Altogether, the PLWK and PLWKr outperformed the other methods in our
preliminary numerical experiments. It is worth mentioning that the
LWKls, due to the line search, demands in each iteration the solution of
three PDE's, while the other methods require the solution of two PDE's
per iteration. In the noisy data case, very soon many residuals drop
bellow the threshold in each cycle, and, in the corresponding iterations,
only one PDE has to be solved (see Figure~\ref{fig:calderon-noisy}).

\section{Final remarks and conclusions} \label{sec:conclusion}

In this article we combine the {\em projective Landweber method} \cite{LS15}
with {\em Kaczmarz's method} \cite{Kac37} for solving systems of non-linear
ill-posed equations.  

The underlying assumption used in convergence analysis presented in this manuscript
is the tangential cone condition \eqref{eq:tcc}.
Notice that the convergence analysis of the PLWK method requires $\eta < 1$ while the
LWK method requires the TCC with $\eta < 0.5$ \cite{HLS07}.

The numerical experiments depicted in Figure~\ref{fig:calderon-noisy} indicate
that, in the noisy data case, the bang-bang relaxation parameter $\omega_k$ in
\eqref{def:wk-def} vanishes for several $k$ (already after the first iterations;
see Figure~\ref{fig:calderon-noisy} Bottom).
Consequently, the computational evaluation of the adjoint $F'_{[k]}(x_k^\delta)^*$
is avoided, making the PLWK and PLWKr methods a fast alternative to conventional
regularization techniques for solving \eqref{eq:single-op} (single equation approach).

The truncation technique used in the Appendix is analogous to the one proposed in
\cite{CHLS08} to prove a similar result for a steepest-descent type method.
The role played by this truncation is merely to provide a sufficient
condition for proving strong convergence of the PLWK method. In the
realistic noisy data case, this truncation does not modify the original
PLWK method introduced in Section~\ref{sec:plwk}, whenever the constant
$\lambda_{max}$ is chosen large enough.

The PLWK and PLWKr methods have proven to be efficient alternatives to
the LWK and LWKls methods for solving ill-posed systems. Comparison with
Newton type methods will be the subject of future work.

\begin{appendices}
\section{\!\!\!\!\!\!ppendix: Strong convergence for exact data} \label{sec:appendix}

In what follows we consider the PLWK iteration in \eqref{eq:plwk-iter} with $\omega_k$
defined as in \eqref{def:wk-def}, and $\tau$, $p_i$ defined as in \eqref{def:tau-p-lbd}.
However, differently from \eqref{def:tau-p-lbd}, $\lambda_k$ is now defined by
\begin{equation}
   \label{def:trunc-lbd}
   \lambda_k := \Lambda\Big( \dfrac{ p_{[k]}( \norm{F_{[k],\delta}(x_k^\delta)} )}
                   { \norm{F_{[k]}'(x_k^\delta)^* F_{[k],\delta}(x_k^\delta)}^2 } \Big) , \
   {\rm if } \ F_{[k]}'(x_k^\delta)^* F_{[k],\delta}(x_k^\delta) \neq 0 , \quad\quad
   \lambda_k := 0 , \ {\rm otherwise.}
\end{equation}
Here $\Lambda: \R^+ \to \R$ is a truncation function satisfying $\Lambda(t) =
\min\{ t, \lambda_{max} \}$ for $t \geq 0$, where $\lambda_{max} > (1-\eta) C^{-2}$
is some positive constant.

In the exact data case we have
$$
p_i(t) \ := \ (1- \eta) \, t^2 , \ \ i \in\set{0,\dots,N-1}
\quad\quad {\rm and } \quad\quad
\omega_k \ := \ \begin{cases}
                 1  & F_{[k],0}(x_k) \neq 0 \\
                 0  & \text{otherwise}
              \end{cases} , \ \ k \in \N .
$$
Moreover, we have either $\lambda_k = 0$ (whenever $F_{[k],0}(x_k) = 0$) or
\begin{equation}
   \label{def:lbd-min}
   \lambda_k \ := \ \min\Big\{ \dfrac{ (1-\eta) \norm{F_{[k],0}(x_k)}^2 }
                    { \norm{F_{[k]}'(x_k)^* F_{[k],0}(x_k)}^2 } \, , \
                    \lambda_{max} \Big\} \ > \ \frac{(1-\eta)}{C^2} \ =: \ \lambda_{min} .
\end{equation}
The inequality in \eqref{def:lbd-min} follows from the fact that $x_k \in B_\rho(x_0)$
for $k \geq 0$, together with assumption \textbf{A1}
(notice that both Proposition~\ref{pr:monot} and Theorem~\ref{th:plw-noise} remain
valid for PLWK with the new definition of $\lambda_k$ in \eqref{def:trunc-lbd}).

In the next theorem we use this setup to prove a strong convergence result for the
PLWK iteration in the case of exact data. The truncation function $\Lambda$ is essential
for obtaining the estimate \eqref{eq:hns-plwk1}.

\begin{theorem}[Strong convergence for exact data] \mbox{} \\
  \label{th:lw.lim.s}
  Let assumptions \textbf{A1} -- \textbf{A3} hold true, $\delta_0 = \dots = \delta_{N-1}
  = 0$ and $(x_k)$ be defined by the PLWK method in \eqref{eq:plwk-iter}, \eqref{def:wk-def}
  with $\lambda_k$ defined as in \eqref{def:trunc-lbd}.
  If \ $\inf$ $\theta_k > 0$ \ and \ $\sup$ $\theta_k < 2$, then $(x_k)$
  converges strongly to some $\bar x \in B_\rho(x_0)$  solving \eqref{eq:inv-probl}.
\end{theorem}

\begin{proof}
We define $e_k := x^\star - x_k$. Since we have exact data, it follows from
Proposition~\ref{pr:monot} that $\norm{e_k}$ is monotone non-increasing.
Thus, $\norm{e_k}$ converges to some $\epsilon \geq 0$. In the following
we show that the sequence $(e_k)$ is a Cauchy sequence.
In order to prove this fact, it suffices to show that
\begin{equation}  \label{eq:csarg}
  |\langle e_l - e_k , e_l \rangle | \to 0   \quad\quad {\rm and } \quad\quad
  |\langle e_l - e_j , e_l \rangle | \to 0
\end{equation}
as $k$, $j \to \infty$, where $k \leq j$ and $l \in \set{k, \dots, j}$ (see, e.g.,
\cite[Theorem~2.3]{HNS95} for the Landweber method or \cite[Theorem~2.3]{HLS07}
for the LWK method).

Let $k \leq j$ be arbitrary. Define \ $k_0 := \lfloor k/N \rfloor$,
$j_0 := \lfloor j/N \rfloor$ and \ $k_1 := [k]$, $j_1 := [j]$. Consequently, \
$k = k_0 N + k_1$, \ $j = j_0 N + j_1$. Now, choose $l_0 \in\set{k_0,\dots,j_0}$
such that
\begin{equation} \label{eq:l-min}
\summ_{n=0}^{N-1} \, \norm{ F_{n,0}(x_{l_0 N + n}) } \ \leq \
\summ_{n=0}^{N-1} \, \norm{ F_{n,0}(x_{i_0 N + n}) }
\end{equation}
for all $i_0 \in \set{k_0, \ldots, j_0}$, \ and set $l := l_0 N + N-1$. Therefore,
\begin{eqnarray}
\lefteqn{ \hspace{-1.0cm} |\langle e_l - e_j , e_l \rangle |
\, = \,
\Big| \summ_{i=l}^{j-1}
\langle (x_{i+1} - x_i), (x^\star - x_l) \rangle \Big|
\, = \,
\Big| \summ_{i=l}^{j-1} \theta_i \, \lambda_i \,
\langle y_{[i]} - F_{[i]}(x_i) ,
        F_{[i]}'(x_i) (x^\star - x_l) \rangle \Big| } \nonumber \\
&\leq&
\summ_{i=l}^{j-1} \theta_i \, \lambda_i \, 
\norm{ F_{[i],0}(x_i) } \,
\norm{ F_{[i]}'(x_i) (x^\star - x_i) \, + \, F_{[i]}'(x_i) (x_i - x_l) } \nonumber \\
&\leq&
2 \, \summ_{i=l}^{j-1} \lambda_i \,
\norm{ F_{[i],0}(x_i) } \, (1+\eta) \,
\Big[ \| F_{[i]}(x^\star) - F_{[i]}(x_i) \| \, + \,
      \| F_{[i]}(x_i) - F_{[i]}(x_l)     \| \Big] \nonumber \\
&=&
2 \, (1+\eta) \, \summ_{i=l}^{j-1} \lambda_i \,
\norm{ F_{[i],0}(x_i) } \,
\Big[ \| F_{[i],0}(x_i) \|
    + \| F_{[i]}(x_i) - y_{[i]} + y_{[i]} - F_{[i]}(x_l) \| \Big] \nonumber \\
&\leq&
2 \, (1+\eta) \summ_{i=l}^{j-1} \lambda_i \,
\norm{ F_{[i],0}(x_i) } \,
\Big[  2 \| F_{[i],0}(x_i) \| \, + \, \| F_{[i]}(x_l) - y_{[i]} \| \Big] \nonumber \\
&=&
4 \, (1+\eta) \summ_{i=l}^{j-1} \lambda_i \, \norm{ F_{[i],0}(x_i) }^2 \ + \
2 \, (1+\eta) \summ_{i=l}^{j-1} \lambda_i \,
\norm{ F_{[i],0}(x_i) } \, \norm{ F_{[i],0}(x_l) } \label{eq:hns-plwk}
\end{eqnarray}
(in the second inequality we used Proposition~\ref{pr:prelim}, item~\ref{it:1}).
Next we estimate the term $\norm{ F_{[i],0}(x_l) }$ on the right hand side
of \eqref{eq:hns-plwk} (to simplify the notation we write $i = i_0N+i_1$, with
$i_i \in\set{0,\dots,N-1}$).
\begin{eqnarray}
\lefteqn{ \hspace{-1.0cm}\norm{ F_{[i],0}(x_l) } \ = \ \norm{ F_{[i]}(x_l) - y_{[i]} }
\ = \ \norm{ F_{i_1}(x_{l_0N+N-1}) - y_{i_1} } \ \leq } \nonumber \\
&\leq&
\norm{ F_{i_1}(x_{l_0N+i_1}) - y_{i_1} } \ + \
\summ_{n=i_1}^{N-2} \norm{ F_{i_1}(x_{l_0N+n+1}) - F_{i_1}(x_{l_0N+n}) } \nonumber \\
&\leq&
\norm{ F_{i_1,0}(x_{l_0N+i_1}) } \ + \ \frac{1}{(1-\eta)} \,
\summ_{n=i_1}^{N-2} \norm{ F'_{i_1}(x_{l_0N+n}) ( x_{l_0N+n+1} - x_{l_0N+n}) } \nonumber \\
&\leq&
\norm{ F_{i_1,0}(x_{l_0N+i_1}) } \ + \ \frac{C}{(1-\eta)} \,
\summ_{n=i_1}^{N-2} \norm{ x_{l_0N+n+1} - x_{l_0N+n} } \nonumber \\
&\leq&
\norm{ F_{i_1,0}(x_{l_0N+i_1}) } \ + \ \frac{C}{(1-\eta)} \,
\summ_{n=i_1}^{N-2} \theta_{l_0N+n} \, \lambda_{l_0N+n} \,
\norm{  F_{n}'(x_{l_0N+n})^* F_{n,0}(x_{l_0N+n}) } \nonumber \\
&\leq&
\norm{ F_{i_1,0}(x_{l_0N+i_1}) } \ + \ \frac{2 C^2}{(1-\eta)} \,
\summ_{n=i_1}^{N-2} \lambda_{max} \, \norm{ F_{n,0}(x_{l_0N+n}) } \nonumber \\
&\leq&
\widetilde{C} \, \summ_{n=i_1}^{N-2} \norm{ F_{n,0}(x_{l_0N+n}) }
\ \leq \
\widetilde{C} \, \summ_{n=0}^{N-1} \norm{ F_{n,0}(x_{l_0N+n}) } \label{eq:hns-plwk1}
\end{eqnarray}
(the second inequality follows from Proposition~\ref{pr:prelim}, item~\ref{it:1}).
Here $\widetilde{C} = \big[ 2(1-\eta) + 4 C^2 \lambda_{max} \big] \, (1-\eta)^{-1}$.
Using \eqref{eq:hns-plwk1} we estimate the second sum on the right hand side of
\eqref{eq:hns-plwk} (once again we adopt the notation $i = i_0N+i_1$).
\begin{eqnarray}
\lefteqn{ \hspace{-2.0cm} \summ_{i=l}^{j-1} \lambda_i \,
\norm{ F_{[i],0}(x_i) } \, \norm{ F_{[i],0}(x_l) } \ \leq \
\summ_{i_0=l_0}^{j_0} \ \summ_{i_1=0}^{N-1} 
\lambda_i \,
\norm{ F_{i_1,0}(x_i) } \, \norm{ F_{i_1,0}(x_l) } } \nonumber \\
&\leq&
\summ_{i_0=l_0}^{j_0} \, \Big[ \, \summ_{i_1=0}^{N-1} 
\, \lambda_i \, \norm{ F_{i_1,0}(x_i) } \,
\Big( \widetilde{C} \summ_{n=0}^{N-1} \, \norm{ F_{n,0}(x_{l_0N+n}) } \Big)
\Big] \nonumber \\
&\leq&
\widetilde{C} \, \lambda_{max} \,
\summ_{i_0=l_0}^{j_0} \, \Big( \summ_{i_1=0}^{N-1} 
\norm{ F_{i_1,0}(x_{i_0N+i_1}) } \Big) \,
\Big( \summ_{n=0}^{N-1} \, \norm{ F_{n,0}(x_{l_0N+n}) } \Big) \nonumber \\
&\leq&
\widetilde{C} \, \lambda_{max} \,
\summ_{i_0=l_0}^{j_0} \, \Big( \summ_{i_1=0}^{N-1} \norm{ F_{i_1,0}(x_{i_0N+i_1}) }
\Big)^2 \nonumber \\
&\leq&
\widetilde{C} \, \lambda_{max} \,
\summ_{i_0=l_0}^{j_0} N \, \summ_{i_1=0}^{N-1} \norm{ F_{i_1,0}(x_{i_0N+i_1}) }^2
\nonumber \\
& = &
\widetilde{C} \, N \, \lambda_{max}
\summ_{i=l_0}^{j_0N+N-1} \, \norm{ F_{[i],0}(x_{i}) }^2, \label{eq:hns-plwk2}
\end{eqnarray}
where the third inequality follows from \eqref{eq:l-min}. Substituting \eqref{eq:hns-plwk2}
in \eqref{eq:hns-plwk} we obtain
\begin{eqnarray*}
|\langle e_l - e_j , e_l \rangle |
& \leq &
         4 (1+\eta)  \lambda_{max} \summ_{i=l}^{j-1} \norm{ F_{[i],0}(x_i) }^2 \ + \
         2 (1+\eta) \widetilde{C} N \lambda_{max} \!\!
\summ_{i=l_0}^{j_0N+N-1} \, \norm{ F_{[i],0}(x_{i}) }^2 \\
& \leq &
         \widetilde{\!\widetilde{C}} \,
         \summ_{i=l_0}^{\infty} \, \lambda_i \, \norm{ F_{[i],0}(x_{i}) }^2 ,
\end{eqnarray*}
where $\widetilde{\!\widetilde{C}} = 2 \lambda_{max} \, (1+\eta) \, [2 + \widetilde{C} N]
\, \lambda_{min}^{-1}$ (in the last inequality we used \eqref{def:lbd-min}).

From \eqref{eq:plw-noise1} and the definition of the index $l \in\set{k,\dots,j}$
it follows that, given $\epsilon > 0$ there exists some $N_\epsilon \in \N$ such that
$|\langle e_l - e_j , e_l \rangle | \leq \epsilon/2$ for $k$, $j \geq N_\epsilon$.
Analogously, one shows that $|\langle e_l - e_k , e_l \rangle | \leq \epsilon$
for $k$, $j \geq N_\epsilon$. This is sufficient to guarantee \eqref{eq:csarg}.

Consequently, $x_k = x^\star - e_k$ converges to some $\bar x \in B_\rho(x_0)$.
Since, due to \eqref{eq:plw-noise1}, the residuals $\norm{F_{[k],0} (x_k)}$ converge
to zero as $k \to \infty$, we conclude that $\bar x$ is a solution of \eqref{eq:inv-probl},
completing the proof.
\end{proof}

%
%

\end{appendices}

\section*{Acknowledgments}

A.L. acknowledges support from the Brazilian research agencies CAPES and
CNPq (grant 309767/13-0).
The work of B.F.S. was partially supported by CNPq (grants 474996/13-1,
302962/11-5) and FAPERJ (grants E-26/102.940/2011, 201.584/2014).

\bibliographystyle{plain}
\bibliography{PLW-Kacz}

\end{document}